\documentclass[11pt, a4paper]{article}
\usepackage{amsmath, amssymb, enumerate, amsthm, natbib, accents}
\usepackage{amsmath, amssymb, url, comment}
\usepackage{graphicx}
\usepackage{tikz}
\usetikzlibrary{arrows,positioning,plotmarks,external,patterns,angles,
	decorations.pathmorphing,backgrounds,fit,shapes,graphs,calc,spy}
\usepackage{pgfplotstable}

\newtheorem{Lemma}{Lemma}

\usepackage{typearea}
\typearea{12}

\makeatletter
\def\widebar{\accentset{{\cc@style\underline{\mskip10mu}}}}
\makeatother

\begin{document}

	\title{Information criteria for non-normalized models}
\author{Takeru~Matsuda\thanks{Statistical Mathematics Unit, RIKEN Center for Brain Science, e-mail: 
		\texttt{takeru.matsuda@riken.jp}},
	~Masatoshi~Uehara\thanks{Department of Computer Science, Cornell University}
	~and~Aapo~Hyv{\"a}rinen\thanks{Department of Computer Science, University of Helsinki}
}
	\date{}
	\maketitle
	
	\begin{abstract}
Many statistical models are given in the form of non-normalized densities with an intractable normalization constant. 
Since maximum likelihood estimation is computationally intensive for these models, 
several estimation methods have been developed which do not require explicit computation of the normalization constant, such as noise contrastive estimation (NCE) and score matching.
However, model selection methods for general non-normalized models have not been proposed so far.
In this study, we develop information criteria for non-normalized models estimated by NCE or score matching.
They are approximately unbiased estimators of discrepancy measures for non-normalized models.
Simulation results and applications to real data demonstrate that the proposed criteria enable selection of the appropriate non-normalized model in a data-driven manner.
	\end{abstract}
		\section{Introduction}\label{sec_intro}
	We consider here the estimation of parametric statistical models which are non-normalized (energy-based). 
	A non-normalized model is one in which the probability density does not integrate to unity. 
	Given a functional form $\widetilde{p}(x \mid \theta)$ for the parametrized density\footnote{In this paper, all measures are assumed to have densities (on $\mathbb{R}^d$ or its subset) and a measure and its density are identified without loss of generality for convenience. We explicitly write input arguments (such as $x,y,z$) of densities for clarity.}, the actual normalized model would be given by
	\begin{align}
		p(x \mid \theta) = \frac{1}{Z(\theta)} \widetilde{p}(x \mid \theta), \label{nnp}
	\end{align}
	where
	\begin{align}
		Z(\theta) = \int \widetilde{p}(x \mid \theta) {\rm d} x.
	\end{align}
	In the framework considered here, estimation of the parameters is attempted without computing the integral defining the normalization constant $Z(\theta)$, which is assumed to be too difficult to compute.
Many statistical models naturally have such property: for instance, Markov random field models \citep{Li}, directional distributions \citep{Mardia,Chikuse}, truncated Gaussian graphical models \citep{Lin}, network models \citep{Caimo}, and energy-based overcomplete independent component analysis models \citep{Teh}.
	Since maximum likelihood estimation is computationally intensive for such non-normalized models,
	several estimation methods have been developed which avoid calculation of the normalization constant.
	These methods include pseudo-likelihood \citep{Besag}, Monte Carlo maximum likelihood \citep{Geyer}, 
	contrastive divergence \citep{Hinton}, score matching \citep{SM} and noise contrastive estimation \citep{Gutmann}. 
	Among them, noise contrastive estimation (NCE) is applicable to general non-normalized models for both continuous and discrete data.
	In NCE, the normalization constant $Z(\theta)$ is estimated together with the unknown parameter $\theta$ by discriminating between data and artificially generated noise.
	On the other hand, score matching is a computationally efficient method for continuous data which is based on a trick of integration by parts.
	The idea of score matching has been generalized to the theory of proper local scoring rules \citep{Parry} and also applied to Bayesian model selection with improper priors \citep{Dawid,Shao}.
	Several studies extended score matching to discrete data \citep{SM2,Lyu}.
	Recently, Stein's method has been applied to estimation of non-normalized models \citep{Barp,Liu}.
	
Although non-normalized models enable more flexible modeling of data-generating processes, information criteria-based model selection methods have not been proposed for NCE and score matching, to the best of our knowledge.
	In general, model selection is the task of selecting a statistical model from several candidates based on data \citep{Burnham,Claeskens,Konishi08}, where different candidates can have different number of parameters.
	By selecting an appropriate model in a data-driven manner, we obtain better understanding of the underlying phenomena and also better prediction of future observations.
	\cite{Akaike} established a unified approach to model selection from the viewpoint of information theory and entropy.
	Specifically, he proposed Akaike Information Criterion (AIC) as a measure of the discrepancy between the true and estimated model in terms of the Kullback--Leibler divergence.
	Thus, the model with the minimum AIC is selected as the best model.
	AIC is widely used in many areas and has been extended by several studies \citep{Takeuchi,Konishi96,Kitagawa97,Spiegelhalter,Watanabe}.
	However, these existing information criteria assume that the model is normalized and thus they are not applicable to non-normalized models.
	
	In this study, we develop information criteria for non-normalized models estimated by NCE or score matching.
	For NCE, based on the observation that NCE is a projection with respect to a Bregman divergence \citep{Hirayama},
	we propose noise contrastive information criterion (NCIC) as an approximately unbiased estimator of the model discrepancy induced by this Bregman divergence. 
	Note that AIC \citep{Akaike} was developed as an approximately unbiased estimator of the Kullback-Leibler discrepancy. 
	Similarly, for score matching, we propose score matching information criterion (SMIC) as an approximately unbiased estimator of the model discrepancy induced by the Fisher divergence \citep{Lyu}. 
	Thus, the non-normalized model with the minimum NCIC or SMIC is selected as the best model.
	Experimental results show that these procedures successfully select the appropriate non-normalized model in a data-driven manner.
	Therefore, this study increases the practicality of non-normalized models.
	Note that \cite{Ji} and \cite{Varin} proposed information criteria based on the pseudo-likelihood and composite likelihood, respectively.
	Whereas their criteria are useful for discrete-valued data, our criteria are applicable to continuous-valued data, and NCIC is equally applicable to discrete-valued data.
	
	This paper is organized as follows.
	In Sections 2 and 3, we briefly review noise contrastive estimation (NCE) and score matching, respectively.
	In Section 4, we review the theory of Akaike information criterion (AIC) and Takeuchi information criterion (TIC).
	In Sections 5 and 6, we develop information criteria for non-normalized models estimated by NCE and score matching, respectively.
	In Section 7, we confirm the validity of NCIC and SMIC by numerical experiments.
	In Section 8, we apply NCIC and SMIC to real data of natural image, RNAseq and wind direction.
	In Section 9, we discuss extension of NCIC to non-normalized mixture models.
	In Section 10, we give concluding remarks.
	
	\section{Noise contrastive estimation (NCE)}\label{sec_NCE}
	In this section, we briefly review noise contrastive estimation (NCE), which is a general method for estimating non-normalized models.
	For more detail, see \cite{Gutmann12}. 
	
	\subsection{Procedure of NCE}
	Suppose we have $N$ i.i.d. samples $x^{(1)},\dots,x^{(N)}$ from a parametric distribution \eqref{nnp}.
	In NCE, we rewrite the non-normalized model \eqref{nnp} to
	\begin{align}
		\log p(x \mid \theta,c) = \log \widetilde{p} (x \mid \theta) + c, \label{NCEparam}
	\end{align}
	where $c=-\log Z(\theta)$.
	We regard $c$ as an additional parameter and estimate it together with $\theta$.
	Note that the final estimate $p(x \mid \hat{\theta},\hat{c})$ is not normalized in general.
	
	In addition to data $x^{(1)},\dots,x^{(N)}$ from the non-normalized model \eqref{nnp}, we generate $M$ i.i.d. noise samples $y^{(1)},\dots,y^{(M)}$ from a noise distribution with density $n(y)$.
	In practice, the noise distribution is usually chosen to be as close as possible to the true data distribution.
	For example, when the data is a random vector, the normal distribution with the same mean and covariance with data is often used as the noise distribution.
	Note that the noise distribution can be non-normalized itself, in which case MCMC can be employed for sampling $y^{(1)},\dots,y^{(M)}$ \citep{Chopin}.
	Then, we estimate $(\theta, c)$ by discriminating between the data and noise as accurately as possible:
	\begin{align}
		(\hat{\theta}_{{\rm NCE}},\hat{c}_{{\rm NCE}}) = {\rm arg} \min_{\theta,c} \hat{d}_{{\rm NCE}} (\theta,c), \label{NCEdef}
	\end{align}
	where 
	\begin{align}
		\hat{d}_{{\rm NCE}} (\theta,c) &= - \frac{1}{N} \sum_{t=1}^N \log \frac{N p(x^{(t)} \mid \theta,c)}{N p(x^{(t)} \mid \theta,c)+M n(x^{(t)})} - \frac{1}{N} \sum_{t=1}^M \log \frac{M n(y^{(t)})}{N p(y^{(t)} \mid \theta,c)+M n(y^{(t)})}. \label{Jdef}
	\end{align}
	The objective function $\hat{d}_{{\rm NCE}}$ is the negative log-likelihood of the logistic regression classifier.
	In other words, each term in $\hat{d}_{{\rm NCE}}$ is the log-probability of the class posterior in a two-class mixture model with a class prior $N $ to $M$ and class distributions $p$ and $n$.
	Note that $\hat{c}_{{\rm NCE}} \neq -\log Z(\hat{\theta}_{{\rm NCE}})$ and so the model $p(x \mid \hat{\theta}_{{\rm NCE}}, \hat{c}_{{\rm NCE}})$ estimated by NCE is not exactly normalized for a finite sample.
	NCE has consistency and asymptotic normality under mild regularity conditions \citep{Gutmann12,Chopin}.
Note that an idea similar to NCE has been employed in the context of biased sampling \citep{Qin}.
	
	\subsection{Bregman divergence related to NCE}
	Here, we explain the observation by \cite{Hirayama} that NCE is interpreted as a projection with respect to a Bregman divergence.
	
	We first review the relationship between the maximum likelihood estimator (MLE) and the Kullback--Leibler divergence.
	For two probability distributions $q(x)$ and $p(x)$, the Kullback--Leibler divergence $D_{{\rm KL}} (q,p)$ and Kullback--Leibler discrepancy $d_{{\rm KL}} (q,p)$ from $q(x)$ to $p(x)$ are defined as
	\begin{align*}
		D_{{\rm KL}} (q,p) = \int q(x) \log \frac{q(x)}{p(x)} {\rm d} x, \quad d_{{\rm KL}} (q,p) = - \int q(x) \log {p(x)} {\rm d} x,
	\end{align*}
	respectively.
	Note that
	\begin{align}
		D_{{\rm KL}} (q,p) = \int q(x) \log {q(x)} {\rm d} x + d_{{\rm KL}} (q,p).
	\end{align}
	For $x^{(1)},\dots,x^{(N)} \sim p(x \mid \theta)$, the MLE is defined as
	\begin{align}
		\hat{\theta}_{{\rm MLE}} = {\rm arg} \max_{\theta} \sum_{t=1}^N \log p(x_t \mid \theta).
	\end{align}
	Let $\hat{q}(x)$ be the empirical distribution of $x^{(1)},\dots,x^{(N)}$ and denote $p(x \mid \theta)$ by $p_{\theta}$.
	Then,
	\begin{align*}
		d_{{\rm KL}} (\hat{q},p_{\theta}) = -\frac{1}{N} \sum_{t=1}^N \log {p(x_t \mid \theta)}.
	\end{align*}
	Therefore, the MLE minimizes the Kullback--Leibler discrepancy between the empirical distribution and the model:
	\begin{align}
		\hat{\theta}_{{\rm MLE}} = {\rm arg} \min_{\theta} {d}_{{\rm KL}} (\hat{q},p_{\theta}).
	\end{align}
	In this sense, the MLE is interpreted as a projection with respect to the Kullback--Leibler divergence.
	
	Now, we present the analogous result for NCE, which is a special case of the general discussion by \cite{Hirayama}.
	Consider a Bregman divergence between two nonnegative measures $q$ and $p$ defined as
	\begin{align}
		D_{{\rm NCE}} (q,p) = \int d_f \left( \frac{q(x)}{n(x)}, \frac{p(x)}{n(x)} \right) n(x) {\rm d} x,
	\end{align}
	where $n(x)$ is a probability density, $d_f (a,b) = f(a) - f(b) - f'(b) (a-b)$ and
	\begin{align}
		f(x) = x \log x - \left( \frac{M}{N}+x \right) \log \left( 1+\frac{N}{M} x \right). \label{Bregman_f}
	\end{align}
	This divergence is decomposed as
	\[
	D_{{\rm NCE}} (q,p) = g(q) + d_{{\rm NCE}} (q,p),
	\]
	where $g(q)$ is a quantity depending only on $q$ and
	\begin{align}
		d_{{\rm NCE}} (q,p) &= -\int q(x) \log \frac{N p(x)}{N p(x) + M n(x)} {\rm d} x - \frac{M}{N} \int n(y) \log \frac{M n(y)}{N p(y) + M n(y)} {\rm d} y. \label{Jqp}
	\end{align}
	Note that $d_{{\rm NCE}} (q,p)=0$ if and only if $q=p$ since $f$ is strictly convex. 
	Then, the objective function $\hat{d}_{{\rm NCE}} (\theta,c)$ of NCE in \eqref{Jdef} satisfies
	\begin{align}
		{\rm E}_y \{ \hat{d}_{{\rm NCE}} (\theta,c) \} = d_{{\rm NCE}} (\hat{q},p_{\theta,c}),
	\end{align}
	where $\hat{q}$ is the empirical distribution of $x^{(1)},\dots,x^{(N)}$, $p_{\theta,c} = p(\cdot \mid \theta,c)$, and ${\rm E}_y$ denotes the expectation with respect to noise samples $y^{(1)},\dots,y^{(M)}$.
	Thus, NCE is interpreted as minimizing the discrepancy $d_{{\rm NCE}} (\hat{q}, p_{\theta,c})$ between the empirical distribution $\hat{q}(x)$ and the model distribution $p(x \mid \theta,c)$.
	Although we can adopt $f$ other than \eqref{Bregman_f}, \cite{Uehara} showed that \eqref{Bregman_f} minimizes the asymptotic variance of the estimator among the class of twice continuously differentiable convex functions.
	
	\section{Score matching}\label{sec_SM}
	In this section, we briefly review the score matching estimator \citep{SM}, which is a computationally efficient estimation method for non-normalized models of continuous data.
	
	The score matching method is based on a divergence called the Fisher divergence \citep{Lyu,Hirayama}.
	For two probability distributions $q$ and $p$ on $\mathbb{R}^d$, the Fisher divergence is defined as
	\[
	D_{{\rm F}} (q,p) = \int \sum_{i=1}^d \left\{ \frac{\partial}{\partial x_i} \log q(x) - \frac{\partial}{\partial x_i} \log p(x) \right\}^2 q(x) {\rm d} x.
	\]
	By using integration by parts, it is transformed as $D_{{\rm F}} (q,p) = g(q) + d_{{\rm SM}} (q,p)$, where $g(q)$ is a quantity depending only on $q$ and
	\begin{align}
		d_{{\rm SM}} (q,p) = \int \left[ 2 \sum_{i=1}^d \frac{\partial^2}{\partial x_i^2} \log p(x)+\sum_{i=1}^d \left\{ \frac{\partial}{\partial x_i} \log p(x) \right\}^2 \right] q(x) {\rm d} x. \label{dSM}
	\end{align}
	
	Now, suppose we have $N$ i.i.d.~samples $x^{(1)},\dots,x^{(N)}$ from an unknown distribution $q(x)$ and fit the non-normalized model \eqref{nnp}.
	Then, an unbiased estimator of $d_{{\rm SM}} (q,p_{\theta})$ in \eqref{dSM} is obtained as
	\begin{align*}
		\hat{d}_{{\rm SM}} (\theta) = \frac{1}{N} \sum_{t=1}^N \rho_{{\rm SM}}(x^{(t)},\theta), 
	\end{align*}
	where
	\begin{align*}
		\rho_{{\rm SM}}(x,\theta) = 2 \sum_{i=1}^d \frac{\partial^2}{\partial x_i^2} \log \widetilde{p}(x \mid \theta)+\sum_{i=1}^d \left\{ \frac{\partial}{\partial x_i} \log \widetilde{p}(x \mid \theta) \right\}^2.
	\end{align*}
	Importantly, we do not need $Z(\theta)$ for computing $\hat{d}_{{\rm SM}} (\theta)$.
	Thus, the score matching estimator is defined as
	\[
	\hat{\theta}_{{\rm SM}} = {\rm arg} \min_{\theta} \hat{d}_{{\rm SM}} (\theta).
	\]
	This estimator has consistency and asymptotic normality under mild regularity conditions \citep{SM}.
	
	\cite{SM2} extended score matching to non-normalized models on  $\mathbb{R}_+^d = [0,\infty)^d$ by considering the divergence
	\[
	D_{{\rm F+}} (q,p) = \int_{\mathbb{R}_+^d} \sum_{i=1}^d \left\{ x_i \frac{\partial}{\partial x_i} \log q(x) - x_i \frac{\partial}{\partial x_i} \log p(x) \right\}^2 q(x) {\rm d} x.
	\]
	Through a similar argument to the original score matching, the score matching estimator for non-negative data is defined as
	\[
	\hat{\theta}_{{\rm SM+}} = {\rm arg} \min_{\theta} \hat{d}_{{\rm SM+}} (\theta),
	\]
	where
	\begin{align*}
		\hat{d}_{{\rm SM+}} (\theta) = \frac{1}{N} \sum_{t=1}^N \rho_{{\rm SM+}}(x^{(t)}, \theta),
	\end{align*}
	\begin{align*}
		\rho_{{\rm SM+}} (x, \theta)= \sum_{i=1}^d \left[ 2 x_i \frac{\partial}{\partial x_i} \log \widetilde{p}(x \mid \theta) + x_i^2 \frac{\partial^2}{\partial x_i^2} \log \widetilde{p}(x \mid \theta) + x_i^2 \left\{ \frac{\partial}{\partial x_i} \log \widetilde{p}(x \mid \theta) \right\}^2 \right]. 
	\end{align*}
	See \cite{Yu18,Yu} for recent developments of non-negative score matching. 
	
	For exponential families, the objective functions of the score matching estimators reduce to quadratic forms \citep{SM2,Forbes}.
	Specifically, for an exponential family
	\begin{align*}
		p(x \mid \theta) = h(x) \exp \left\{ \sum_{k=1}^m \theta_k T_k(x) - \psi(\theta) \right\} 
	\end{align*}
	on $\mathbb{R}^d$ or $\mathbb{R}_+^d$, the function $\rho_{{\rm SM}} (x,\theta)$ or $\rho_{{\rm SM+}} (x,\theta)$ is given by a quadratic form
	\begin{align}
		\frac{1}{2} \theta^{\top} \Gamma(x) \theta + g(x)^{\top} \theta + c(x). \label{exp_SM}
	\end{align}
	For the exact forms of $\Gamma(x)$, $g(x)$ and $c(x)$, see \cite{Lin}\footnote{Note that $x_{ij}^2$ is missing in the first term of (2.15) in \cite{Lin}.}.
	Thus, the score matching estimator is obtained by solving the linear equation $\left\{ \sum_{t=1}^N \Gamma(x^{(t)}) \right\} \hat{\theta} + \sum_{t=1}^N g(x^{(t)}) = 0$. 
	
	\section{Akaike information criterion (AIC)}\label{sec_AIC}
	In this section, we briefly review the theory of Akaike information criterion.
	For more details, see \cite{Burnham} and \cite{Konishi08}.
	
	Suppose we have $N$ independent and identically distributed (i.i.d.) samples $x^N=(x^{(1)},\dots,x^{(N)})$ from an unknown distribution $q(x)$.
	Based on them, we predict the future observation $z$ from $q(z)$ by using a predictive distribution.
	For this aim, we assume a parametric distribution $p(x \mid \theta)$ with an unknown parameter $\theta \in \mathbb{R}^k$ and construct a predictive distribution $p(z \mid \hat{\theta}_{{\rm MLE}}(x^N))$, where $\hat{\theta}_{{\rm MLE}} (x^N)$ is the maximum likelihood estimate of $\theta$ from $x^N$.
	Then, the distance between the true distribution $q(z)$ and the predictive distribution $p(z \mid \hat{\theta}_{{\rm MLE}}(x^N))$ is evaluated by the Kullback--Leibler divergence
	\begin{align*}
		D_{{\rm KL}} \{ q,\hat{\theta}_{{\rm MLE}} (x^N) \} &= \int q(z) \log \frac{q(z)}{p \{ z \mid \hat{\theta}_{{\rm MLE}} (x^N) \}} {\rm d} z.
	\end{align*}
	The Kullback--Leibler divergence is decomposed as
	\begin{align}
		D_{{\rm KL}} \{ q,\hat{\theta}_{{\rm MLE}} (x^N) \} &= {\rm E}_z \{ \log q(z) \} + d_{{\rm KL}} \{ q,\hat{\theta}_{{\rm MLE}} (x^N) \}, \label{KLdecomp}
	\end{align}
	where ${\rm E}_z$ denotes the expectation with respect to $z \sim q(z)$ and $d_{{\rm KL}} \{ q,\hat{\theta}_{{\rm MLE}} (x^N) \} = - {\rm E}_z [\log p \{ z \mid \hat{\theta}_{{\rm MLE}} (x^N) \}]$ is the Kullback--Leibler discrepancy from the true distribution $q(z)$ to the predictive distribution $p \{ z \mid \hat{\theta}_{{\rm MLE}}(x^N) \}$.
	Since the first term ${\rm E}_z \{ \log q(z) \}$ in \eqref{KLdecomp} does not depend on $\hat{\theta}_{{\rm MLE}} (x^N)$, information criteria are developed as approximately unbiased estimators of the expected Kullback--Leibler discrepancy ${\rm E}_x [d_{{\rm KL}} \{ q,\hat{\theta}_{{\rm MLE}} (x^N) \} ]$, where ${\rm E}_x$ denotes the expectation with respect to $x^{(1)},\dots,x^{(N)} \sim q(x)$.
	
	Let $\hat{q}$ be the empirical distribution of $x^{(1)},\dots,x^{(N)}$.
	Then, the quantity 
	\begin{align}
		d_{{\rm KL}} \{ \hat{q},\hat{\theta}_{{\rm MLE}}(x^N) \} = -\frac{1}{N} \sum_{t=1}^N \log p \{ x^{(t)} \mid \hat{\theta}_{{\rm MLE}}(x^N) \} \label{crude}
	\end{align}
	can be considered as an estimator of ${\rm E}_x [d_{{\rm KL}} \{ q,\hat{\theta}_{{\rm MLE}} (x^N) \} ]$.
	However, this simple estimator has negative bias, because the maximum likelihood estimate $\hat{\theta}_{{\rm MLE}}(x^N)$ is defined to minimize $d_{{\rm KL}}(\hat{q},\theta)$:
	\begin{align}
		\hat{\theta}_{{\rm MLE}} (x^N) = {\rm arg} \min_{\theta} d_{{\rm KL}}(\hat{q}, \theta). \label{mle_proj}
	\end{align}
	If we can compute this negative bias and remove it, then we can actually obtain an unbiased estimator of the Kullback--Leibler discrepancy. 
	This is how typical information criteria are constructed, correcting the inherent bias in using MLE for estimating the Kullback--Leibler discrepancy.
	
	Let $\theta^* = {\rm arg} \min_{\theta} d_{{\rm KL}}(q,\theta)$.
	By putting $D_1 = d_{{\rm KL}} \{ \hat{q},\hat{\theta}_{{\rm MLE}}(x^N) \} - d_{{\rm KL}} \{ \hat{q},{\theta}^* \}$, $D_2 = d_{{\rm KL}} \{ \hat{q},{\theta}^* \} - d_{{\rm KL}} \{ {q},{\theta}^* \}$ and $D_3 = d_{{\rm KL}} \{ {q},{\theta}^* \} - d_{{\rm KL}} \{ q,\hat{\theta}_{{\rm MLE}} (x^N) \}$, we have
	\begin{align}
		d_{{\rm KL}} \{ \hat{q},\hat{\theta}_{{\rm MLE}}(x^N) \} - d_{{\rm KL}} \{ q,\hat{\theta}_{{\rm MLE}} (x^N) \} = D_1 + D_2 + D_3.  \label{dKLdecomp}
	\end{align}
	By definition, ${\rm E}_x (D_2) = 0$.
	Also, as $N \to \infty$,
	\begin{align}
		N D_1 \stackrel{d}{\longrightarrow} -\frac{1}{2} s_1^{\top} J(\theta^*) s_1, \quad N D_3 \stackrel{d}{\longrightarrow} -\frac{1}{2} s_3^{\top} J(\theta^*) s_3, \label{mle_bias}
	\end{align}
	where $s_1 \sim {\rm N} \left\{ 0,J(\theta^*)^{-1} I(\theta^*) J(\theta^*)^{-1} \right\}$, $s_3 \sim {\rm N} \left\{ 0,J(\theta^*)^{-1} I(\theta^*) J(\theta^*)^{-1} \right\}$, $k \times k$ matrices $I(\theta)$ and $J(\theta)$ are defined as
	\begin{align*}
		{I}_{ij}(\theta) &= {\rm E}_z \left\{ \frac{\partial}{\partial \theta_i} \log p(z \mid \theta) \frac{\partial}{\partial \theta_j} \log p(z \mid \theta) \right\}, \\ 
		{J}_{ij}(\theta) &= - {\rm E}_z \left\{ \frac{\partial^2}{\partial \theta_i \partial \theta_j} \log p(z \mid \theta) \right\},
	\end{align*}
	and $J(\theta^*)$ is assumed to be positive definite.
	Note that the expectation of the limit distribution of $N D_1$ and $N D_3$ is
	\begin{align}
		-\frac{1}{2} {\rm E} \{ s_1^{\top} J(\theta^*) s_1 \} = -\frac{1}{2} {\rm tr} \left\{ I(\theta^*) J(\theta^*)^{-1} \right\}. \label{trIJ0}
	\end{align}
	
	From \eqref{crude}, \eqref{dKLdecomp}, \eqref{mle_bias} and \eqref{trIJ0}, \cite{Takeuchi} proposed
	\begin{align}
		{\rm TIC} = -2 \sum_{t=1}^N \log p \{ x^{(t)} \mid \hat{\theta}_{{\rm MLE}}(x^N) \} + 2 {\rm tr} (\hat{I} \hat{J}^{-1}) \label{TIC}
	\end{align}
	as an approximately unbiased estimator of $2 N {\rm E}_x [d_{{\rm KL}} \{ q,\hat{\theta}_{{\rm MLE}} (x^N) \}]$, where $\hat{I}$ and $\hat{J}$ are consistent estimators of $I(\theta^*)$ and $J(\theta^*)$ given by
	\begin{align*}
		\hat{I}_{ij} = \frac{1}{N} \left. \sum_{t=1}^N \frac{\partial}{\partial \theta_i} \log p(x^{(t)} \mid \theta) \frac{\partial}{\partial \theta_j} \log p(x^{(t)} \mid \theta) \right|_{\theta=\hat{\theta}_{{\rm MLE}}(x^N)},
	\end{align*}
	\begin{align*}
		\hat{J}_{ij} = -\frac{1}{N} \left. \sum_{t=1}^N \frac{\partial^2}{\partial \theta_i \partial \theta_j} \log p(x^{(t)} \mid \theta) \right|_{\theta=\hat{\theta}_{{\rm MLE}}(x^N)}.
	\end{align*}
	The quantity \eqref{TIC} is called Takeuchi Information Criterion.
	
	If the model includes the true distribution: $q(x)=p(x \mid \theta^*)$ for some $\theta^*$, then $I(\theta^*)$ and $J(\theta^*)$ coincide and thus ${\rm tr} \{ I(\theta^*) J(\theta^*)^{-1} \} = k$.
	Recall that $k$ is the dimension of the parameter $\theta \in \mathbb{R}^k$.
	Based on this, \cite{Akaike} proposed
	\begin{align}
		{\rm AIC} = -2 \sum_{t=1}^N \log p \{ x^{(t)} \mid \hat{\theta}_{{\rm MLE}}(x^N) \} + 2 k \label{AIC}
	\end{align}
	as an approximately unbiased estimator of $2 N {\rm E}_x [d_{{\rm KL}} \{ q,\hat{\theta}_{{\rm MLE}} (x^N) \}]$.
	The quantity \eqref{AIC} is called Akaike Information Criterion.
	
	Thus, information criteria enable to compare the goodness of fit of statistical models.
	Among several candidate models, the model with minimum information criterion is considered to be the closest to the true data-generating process.
	In practice, since TIC requires more computation than AIC and, furthermore, TIC often suffers from instability caused by estimation errors in $\hat{I}$ and $\hat{J}$, AIC is recommended to use as long as the model is not badly mis-specified \citep[see][Section 2.3]{Burnham}.

	\section{Information criteria for NCE (NCIC)}\label{sec_NCIC}
	In this section, we develop new information criteria for NCE, which we call the Noise Contrastive Information Criterion (NCIC).

	\subsection{Setting and assumptions}
	Suppose we have $N$ i.i.d.~samples $x^{(1)},\dots,x^{(N)}$ from an unknown distribution $q(x)$ and estimate a non-normalized model \eqref{NCEparam} by using NCE with $M$ noise samples $y^{(1)},\dots,y^{(M)}$ from $n(y)$.
	The true distribution $q(x)$ may not be contained in the assumed non-normalized model.
	The parameter $\theta$ in the non-normalized model \eqref{NCEparam} is assumed to be identifiable and have a compact parameter space $\Theta \subset \mathbb{R}^{m-1}$.
	
	For convenience, we denote $\xi=(\theta,c)$, $m={\rm dim}(\xi)={\rm dim}(\theta)+1$, $\hat{\xi}=\hat{\xi}_{{\rm NCE}}$ and $\hat{p}(x)=p(x \mid \hat{\xi})$.
	Namely, $\hat{\xi}=\hat{\xi}_{{\rm NCE}}$ is the minimizer of $\hat{d}_{{\rm NCE}}$ in \eqref{Jdef}.
	The gradient and Hessian with respect to $\xi$ are written as $\nabla_{\xi}$ and $\nabla_{\xi}^2$, respectively.
	Also, we define $\xi^* = {\rm arg} \min_{\xi} d_{{\rm NCE}} (q,p_{\xi}) = (\theta^{*},c^*)$ and write $p_*(x)=p (x \mid \xi^*)$.
	Note that $p_*(x)=q(x)$ when the model includes the true distribution.
	We denote the expectation with respect to $x^{(1)},\dots,x^{(N)} \sim q(x)$ and $y^{(1)},\dots,y^{(M)} \sim n(y) $ by ${\rm E}_{x,y}$.
	The expectation and covariance matrix with respect to $z \sim p(z)$ are denoted by ${\rm E}_p$ and ${\rm Cov}_p$, respectively.
	
	Following \cite{Gutmann12} and \cite{Chopin}, we consider the asymptotics where $N \to \infty$, $M \to \infty$ and $M/N \to \nu$ with $0<\nu<\infty$.
	Let
	\begin{align}
		\rho_d (x,\xi) &= -\log \frac{N p(x \mid \xi)}{N p(x \mid \xi) + M n(x)}, \label{rhod} \\
		\rho_n (y,\xi) &= -\log \frac{M n(y)}{N p(y \mid \xi) + M n(y)}. \label{rhon}
	\end{align}
	Then, the objective function to be minimized in NCE is represented as
	\begin{align*}
		\hat{d}_{{\rm NCE}}(\xi) = \frac{1}{N} \sum_{t=1}^N \rho_d(x^{(t)},\xi) + \frac{1}{N} \sum_{t=1}^M \rho_n (y^{(t)},\xi).
	\end{align*}
	Define $m \times m$ matrices ${I}(\xi)$ and ${J}(\xi)$ by 
	\begin{align*}
		I(\xi) =& \frac{N}{N+M} {\rm Cov}_q \left\{ \nabla_{\xi} \rho_d (z,\xi) \right\} + \frac{M}{N+M} {\rm Cov}_n \left\{ \nabla_{\xi} \rho_n (z,\xi) \right\},
	\end{align*}
	\begin{align*}
		J(\xi) =& \frac{N}{N+M} {\rm E}_q \left\{ \nabla_{\xi}^2 \rho_d (z,\xi) \right\} + \frac{M}{N+M} {\rm E}_n \left\{ \nabla_{\xi}^2 \rho_n (z,\xi) \right\}. 
	\end{align*}
	We assume the following regularity conditions:
	
	\begin{itemize}
		\item[(N1)] {For every $\theta$, the support of $p(z \mid \theta)$ is included in that of $n(z)$.}
		
		\item[(N2)] {For every $z$, $\log p(z \mid \theta)$ is three times continuously differentiable over $\Theta$.}
		
		\item[(N3)] {Both ${\rm  E}_q \left\{ \nabla_{\xi} \rho_d(z,\xi^*) \nabla_{\xi} \rho_d(z,\xi^*)^{\top} \right\}$ and ${\rm  E}_n \left\{ \nabla_{\xi} \rho_n(z,\xi^*) \nabla_{\xi} \rho_n(z,\xi^*)^{\top} \right\}$ are finite.}
		
		\item[(N4)] {There exist functions $b_d(z)$ and $b_n(z)$ such that
			\begin{align*}
				\left| \rho_d(z,\xi) \right| \leq b_d(z), \quad \left| \frac{\partial^2}{\partial \xi_i \partial \xi_j} \rho_d(z,\xi) \right| \leq b_d(z), \quad \left| \frac{\partial^3}{\partial \xi_i \partial \xi_j \partial \xi_k} \rho_d(z,\xi) \right| \leq b_d(z), \\
				\left| \rho_n(z,\xi) \right| \leq b_n(z), \quad \left| \frac{\partial^2}{\partial \xi_i \partial \xi_j} \rho_n(z,\xi) \right| \leq b_n(z), \quad \left| \frac{\partial^3}{\partial \xi_i \partial \xi_j \partial \xi_k} \rho_n(z,\xi) \right| \leq b_n(z)
			\end{align*}
			for all $i,j,k$ and $\xi$, where ${\rm E}_q \{ b_d(z) \} < \infty$ and ${\rm E}_n \{ b_n(z) \} < \infty$.}
		
		\item[(N5)] {The matrix $J(\xi^*)$ is nonsingular.}
	\end{itemize}

	Assumption (N1) is standard in NCE \citep{Gutmann12}.
	Assumptions (N2)-(N5) are similar to the regularity conditions for AIC and TIC \citep{Konishi08}.
	Note that $d_{{\rm NCE}}(q, p_*)$ is finite from the assumption (N4).
	
	\subsection{Bias evaluation}\label{subsec:bias_NCE}
	Similarly to $d_{{\rm KL}} \{ \hat{q},\hat{\theta}_{{\rm MLE}}(x^N) \}$ in \eqref{crude}, the quantity $\hat{d}_{{\rm NCE}} (\hat{\xi})$ has negative bias as an estimator of ${\rm E}_{x,y} \{ d_{{\rm NCE}}(q,\hat{p}) \}$.
	Here, we evaluate this bias following a similar argument to AIC and TIC \citep{Burnham,Konishi08}.
	
	First, the asymptotic distribution of NCE is obtained as follows.
	
	\begin{Lemma}\label{lem_asymp} 
		Under (N1)-(N5),
		\begin{align}
			\sqrt{N} \left( \hat{\xi} - \xi^* \right) \stackrel{d}{\longrightarrow} {\rm N} \left\{ 0,J(\xi^*)^{-1} I(\xi^*) J(\xi^*)^{-1} \right\}. \label{nce_asymp}
		\end{align}
	\end{Lemma}
	\begin{proof}
		The current setting of NCE corresponds to stratified sampling with two strata: data (size $N$) and noise (size $M$).
		From Theorem 3.1 of \cite{Wooldridge}, $\hat{\xi}$ is consistent: $\hat{\xi} \stackrel{p}{\longrightarrow} \xi^*$.
		Then, from Theorem 3.2 of \cite{Wooldridge}, the asymptotic distribution of $\hat{\xi}$ is obtained as \eqref{nce_asymp}.
	\end{proof}
	
	Note that \eqref{nce_asymp} is valid even when the model is mis-specified.
	We also note that \cite{Chopin} established a rigorous asymptotic theory of NCE under general MCMC sampling of noise.
	
	Let $D_1 = \hat{d}_{{\rm NCE}}(\hat{\xi}) - \hat{d}_{{\rm NCE}}({\xi}^*)$, $D_2 = \hat{d}_{{\rm NCE}}({\xi}^*) - {d}_{{\rm NCE}}(q,p_{*})$ and $D_3 = {d}_{{\rm NCE}}(q,p_*) - {d}_{{\rm NCE}}(q,\hat{p})$.
	Then,
	\begin{align}
		\hat{d}_{{\rm NCE}}(\hat{\xi}) - d_{{\rm NCE}}(q,\hat{p}) = D_1+D_2+D_3. \label{NCEdecomp}
	\end{align}
	
	\begin{Lemma}\label{lem_bias}
		Under (N1)-(N5),
		\begin{itemize}
			\item[(a)] $N D_1 \stackrel{d}{\longrightarrow} -\frac{1}{2} s^{\top} J(\xi^*) s$, where $s \sim {\rm N} \left\{ 0,J(\xi^*)^{-1} I(\xi^*) J(\xi^*)^{-1} \right\}$.
			
			\item[(b)] ${\rm E}_{x,y} (D_2)=0$.
			
			\item[(c)] $N D_3 \stackrel{d}{\longrightarrow} -\frac{1}{2} s^{\top} J(\xi^*) s$, where $s \sim {\rm N} \left\{ 0,J(\xi^*)^{-1} I(\xi^*) J(\xi^*)^{-1} \right\}$.
		\end{itemize}
	\end{Lemma}
	\begin{proof}
		(a)
		Since $\nabla_{\xi} \hat{d}_{{\rm NCE}}(\xi)=0$ at $\xi=\hat{\xi}$, the Taylor expansion of $\hat{d}_{{\rm NCE}}(\xi)$ around $\xi=\hat{\xi}$ is given by
		\begin{align*}
			&\hat{d}_{{\rm NCE}}(\xi^*) = \hat{d}_{{\rm NCE}}(\hat{\xi}) + \frac{1}{2} (\xi^{*}-\hat{\xi})^{\top} \nabla_{\xi}^2 \hat{d}_{{\rm NCE}}(\xi^{\dagger}) (\xi^{*}-\hat{\xi}),
		\end{align*}
		where $\xi^{\dagger}$ is a vector on the segment from $\hat{\xi}$ to $\xi^*$.
		From $\hat{\xi} \stackrel{p}{\longrightarrow} \xi^*$ and the discussion in Section 3.3 of \cite{Wooldridge}, $\nabla_{\xi}^2 \hat{d}_{{\rm NCE}}(\xi^{\dagger}) \stackrel{p}{\longrightarrow} J(\xi^*)$.
		Therefore, from Lemma~\ref{lem_asymp} and Slutsky's theorem,
		\begin{align*}
			N D_1 = -\frac{N}{2} (\hat{\xi}-\xi^*)^{\top} \nabla_{\xi}^2 \hat{d}_{{\rm NCE}}(\xi^{\dagger}) (\hat{\xi}-\xi^*) \stackrel{d}{\longrightarrow} -\frac{1}{2} s^{\top} J(\xi^{*}) s, 
		\end{align*}
		where $s \sim {\rm N} \left\{ 0,J(\xi^*)^{-1} I(\xi^*) J(\xi^*)^{-1} \right\}$.
		
		(b)
		From \eqref{Jdef}, \eqref{Jqp} and the law of large numbers, ${\rm E}_{x,y} (D_2)=0$. 
		
		(c)
		Since $\nabla_{\xi} {d}_{{\rm NCE}}(q,p_{\xi})=0$ at $\xi=\xi^*$ and $\nabla_{\xi}^2 {d}_{{\rm NCE}}(q,p_{\xi})=J(\xi)$,  the Taylor expansion of ${d}_{{\rm NCE}}(q,p_{\xi})$ around $\xi=\xi^*$ is given by
		\begin{align*}
			&{d}_{{\rm NCE}}(q,\hat{p}) = {d}_{{\rm NCE}}(q,p_*) + \frac{1}{2} (\hat{\xi}-\xi^{*})^{\top} J (\xi^{\dagger}) (\hat{\xi}-\xi^{*}),
		\end{align*}
		where $\xi^{\dagger}$ is a vector on the segment from $\hat{\xi}$ to $\xi^*$.
		From $\hat{\xi} \stackrel{p}{\longrightarrow} \xi^*$ and the continuous mapping theorem, we have $J(\xi^{\dagger}) = J(\xi^*)+o_p(1)$.
		Therefore, from Lemma~\ref{lem_asymp} and Slutsky's theorem,
		\begin{align*}
			N D_3 = -\frac{N}{2} (\hat{\xi}-\xi^*)^{\top} J(\xi^{\dagger}) (\hat{\xi}-\xi^*) \stackrel{d}{\longrightarrow} -\frac{1}{2} s^{\top} J(\xi^{*}) s, 
		\end{align*}
		where $s \sim {\rm N} \left\{ 0,J(\xi^*)^{-1} I(\xi^*) J(\xi^*)^{-1} \right\}$.
	\end{proof}
	
	From Lemma~\ref{lem_bias}, the expectation of the limit distribution of $N D_1$ and $N D_3$ is
	\begin{align}
		-\frac{1}{2} {\rm E} \{ s^{\top} J(\xi^*) s \} = -\frac{1}{2} {\rm tr} \left\{ I(\xi^*) J(\xi^*)^{-1} \right\}. \label{trIJ}
	\end{align}
	
	When the model includes the true distribution (well-specified case), \eqref{trIJ} has a simpler form.
	Let
	\begin{align*}
		b(z) = \frac{p_*(z) n(z)}{r(z)^2}, 
	\end{align*}
	where
	\begin{align}
		r(z)=\frac{N}{N+M} {p}_*(z)+ \frac{M}{N+M} n(z). \label{rdef}
	\end{align}
	is a mixture distribution of ${p}_*$ and $n$.
	
	\begin{Lemma}\label{lem_bias2}
		Assume that the model includes the true distribution: $q(x)=p(x \mid \xi^*)$. Then,
		\begin{align}
			{\rm tr} \left\{ I(\xi^*) J(\xi^*)^{-1} \right\} &= m- {\rm E}_r \left\{ b(z) \right\}, \label{lem2_eq}
		\end{align}
		where ${\rm E}_r$ denotes the expectation with respect to $z \sim r(z)$ in \eqref{rdef}.
	\end{Lemma}
	\begin{proof}
		Let $s(z \mid \xi) = \nabla_{\xi} \log p(z \mid \xi)$ and $H(z \mid \xi) = \nabla_{\xi}^2 \log p(z \mid \xi)$.
		Since $\log p(z \mid \xi) = \log \widetilde{p}(z \mid \theta) + c$ where $\xi=(\theta,c)$, we have $s_m(z \mid \xi) = 1$ and $H_{im}(z \mid \xi)=H_{mi}(z \mid \xi) = 0$ for $i=1,\dots,m$.
		
		From the definition of $\rho_d$ and $\rho_n$ in \eqref{rhod} and \eqref{rhon},
		\begin{align*}
			\nabla_{\xi} \rho_d (z,\xi) = -\frac{M n(z)}{N p(z \mid \xi) + M n(z)} s(z \mid \xi), \quad 
			\nabla_{\xi} \rho_n (z,\xi) = \frac{N p(z \mid \xi)}{N p(z \mid \xi) + M n(z)} s(z \mid \xi).
		\end{align*}
		Thus,
		\begin{align*}
			\nabla_{\xi}^2 \rho_d (z,\xi) = \frac{N p(z \mid \xi) \cdot M n(z)}{(N p(z \mid \xi) + M n(z))^2} s(z \mid \xi) s(z \mid \xi)^{\top} -\frac{M n(z)}{N p(z \mid \xi) + M n(z)} H(z \mid \xi), \\
			\nabla_{\xi}^2 \rho_n (z,\xi) = \frac{N p(z \mid \xi) \cdot M n(z)}{(N p(z \mid \xi) + M n(z))^2} s(z \mid \xi) s(z \mid \xi)^{\top} +\frac{N p(z \mid \xi)}{N p(z \mid \xi) + M n(z)} H(z \mid \xi).
		\end{align*}
		Therefore,
		\begin{align*}
			J(\xi^*) &= \frac{N}{N+M} \int p(z \mid \xi^*) \nabla_{\xi}^2 \rho_d (z,\xi^*) {\rm d} z + \frac{M}{N+M} \int n(z) \nabla_{\xi}^2 \rho_n (z,\xi^*) {\rm d} z \\
			&= \frac{1}{N+M} \int \frac{N p(z \mid \xi^*) \cdot M n(z)}{N p(z \mid \xi^*) + M n(z)} s(z \mid \xi^*) s(z \mid \xi^*)^{\top} {\rm d} z.
\end{align*}
		Since $s_m(z \mid \xi) = 1$, the $m$-th column vector of $J(\xi^*)$ is
		\begin{align*}
			j_m(\xi^*) = \frac{1}{N+M} \int \frac{N p(z \mid \xi^*) \cdot M n(z)}{N p(z \mid \xi^*) + M n(z)} s(z \mid \xi^*) {\rm d} z.
		\end{align*}
		Then, 
		\begin{align*}
			I(\xi^*) &= \frac{N}{N+M} {\rm Cov}_q \left\{ \nabla_{\xi} \rho_d (z,\xi^*) \right\} + \frac{M}{N+M} {\rm Cov}_n \left\{ \nabla_{\xi} \rho_n (z,\xi^*) \right\} \\
			&= J(\xi^*) - \frac{(N+M)^2}{NM} j_m(\xi^*) j_m(\xi^*)^{\top},
		\end{align*}
		where we used
		\begin{align*}
			&{\rm Cov}_q \left\{ \nabla_{\xi} \rho_d (z,\xi) \right\} \\
			=& {\rm E}_q \left\{ \nabla_{\xi} \rho_d (z,\xi) \nabla_{\xi} \rho_d (z,\xi)^{\top} \right\} - {\rm E}_q \left\{ \nabla_{\xi} \rho_d (z,\xi) \right\} {\rm E}_q \left\{ \nabla_{\xi} \rho_d (z,\xi) \right\}^{\top} \\
			=& \int p(z \mid \xi) \left( \frac{M n(z)}{N p(z \mid \xi) + M n(z)} \right)^2 s(z \mid \xi) s(z \mid \xi)^{\top} {\rm d} z \\
			& \quad - \left( \int p(z \mid \xi) \frac{M n(z)}{N p(z \mid \xi) + M n(z)} s(z \mid \xi) {\rm d} z \right) \left( \int p(z \mid \xi) \frac{M n(z)}{N p(z \mid \xi) + M n(z)} s(z \mid \xi) {\rm d} z \right)^{\top},
		\end{align*}
		\begin{align*}
			&{\rm Cov}_n \left\{ \nabla_{\xi} \rho_n (z,\xi) \right\} \\
			=& {\rm E}_n \left\{ \nabla_{\xi} \rho_n (z,\xi) \nabla_{\xi} \rho_n (z,\xi)^{\top} \right\} - {\rm E}_n \left\{ \nabla_{\xi} \rho_n (z,\xi) \right\} {\rm E}_q \left\{ \nabla_{\xi} \rho_n (z,\xi) \right\}^{\top} \\
			=& \int n(z) \left( \frac{N p(z \mid \xi)}{N p(z \mid \xi) + M n(z)} \right)^2 s(z \mid \xi) s(z \mid \xi)^{\top} {\rm d} z \\
			& \quad - \left( \int n(z) \frac{N p(z \mid \xi)}{N p(z \mid \xi) + M n(z)} s(z \mid \xi) {\rm d} z \right) \left( \int n(z) \frac{N p(z \mid \xi)}{N p(z \mid \xi) + M n(z)} s(z \mid \xi) {\rm d} z \right)^{\top}.
		\end{align*}
		Thus,
		\begin{align*}
			{\rm tr} \left\{ I(\xi^*) J(\xi^*)^{-1} \right\} &= m-\frac{(N+M)^2}{NM} j_m(\xi^*)^{\top} J(\xi^*)^{-1} j_m(\xi^*) \nonumber \\
			&= m- {\rm E}_r \left\{ b(z) \right\}. 
		\end{align*}
	\end{proof}
	
	\cite{Gutmann12} pointed out that NCE converges to the maximum likelihood estimator as $M/(N+M) \to 1$ and \cite{Chopin} gave its proof.
	In this setting, $r(z)$ converges to $n(z)$ and thus ${\rm E}_r \left\{ b(z) \right\}$ goes to one.
	As a result, \eqref{lem2_eq} goes to $m-1$, which is equal to the dimension of the parameter $\theta$.
	
	\cite{Pan} and \cite{Mattheou} proposed information criteria with the quasi-likelihood and density power divergence, respectively, based on similar bias calculations.
	In comparison, the bias term here takes a more complicated form because we estimate not only the parameter but also the normalization constant in NCE.
	
	\subsection{Noise Contrastive Information Criterion (NCIC)}
	Now, we develop NCIC by using the bias evaluation in the previous subsection.
	
	Let
	\begin{align*}
		\overline{\nabla_{\xi} \rho_d} = \frac{1}{N} \sum_{t=1}^N \nabla_{\xi} \rho_d(x^{(t)},\hat{\xi}), \\ \overline{\nabla_{\xi} \rho_n} = \frac{1}{M} \sum_{t=1}^M \nabla_{\xi} \rho_n(y^{(t)},\hat{\xi}),
	\end{align*}
	and define $m \times m$ matrices $\hat{I}$ and $\hat{J}$ by
	\begin{align*}
		\hat{I} = \frac{1}{N+M} & \left[ \sum_{t=1}^N \left\{ \nabla_{\xi} \rho_d (x^{(t)},\hat{\xi})-\overline{\nabla_{\xi} \rho_d} \right\} \left\{ \nabla_{\xi} \rho_d (x^{(t)},\hat{\xi})-\overline{\nabla_{\xi} \rho_d} \right\}^{\top} \right. \\
		&\left. + \sum_{t=1}^M \left\{ \nabla_{\xi} \rho_n (y^{(t)},\hat{\xi})-\overline{\nabla_{\xi} \rho_n} \right\} \left\{ \nabla_{\xi} \rho_n (y^{(t)},\hat{\xi})-\overline{\nabla_{\xi} \rho_n} \right\}^{\top} \right],
	\end{align*}
	\begin{align*}
		\hat{J} = \frac{1}{N+M} \left\{ \sum_{t=1}^N \nabla_{\xi}^2 \rho_d (x^{(t)},\hat{\xi}) + \sum_{t=1}^M \nabla_{\xi}^2 \rho_n (y^{(t)},\hat{\xi}) \right\}. 
	\end{align*}
	From the discussion in Section 3.3 of \cite{Wooldridge}, $\hat{I}$ and $\hat{J}$ are consistent estimators of $I(\xi^*)$ and $J(\xi^*)$, respectively.
	Thus, from \eqref{NCEdecomp} and Lemma \ref{lem_bias}, we propose the quantity
	\begin{align}
		{\rm NCIC}_1 = N \hat{d}_{{\rm NCE}} (\hat{\xi}_{{\rm NCE}}) + {\rm tr} (\hat{I} \hat{J}^{-1}) \label{NCIC0}
	\end{align}
	as an approximately unbiased estimator of $N {\rm E}_{x,y} \left\{ d_{{\rm NCE}} (q,\hat{p}) \right\}$.
	
	We also propose a simpler version of NCIC by assuming that the model includes the true distribution.
	Let
	\begin{align}
		\hat{b}(z) = \frac{\hat{p}(z) n(z)}{\hat{r}(z)^2}, \label{b_def}
	\end{align}
	where
	\begin{align}
		\hat{r}(z)=\frac{N}{N+M} \hat{p}(z)+ \frac{M}{N+M} n(z). \end{align}
	Then, from Lemma \ref{lem_bias2}, we propose the quantity
	\begin{align}
		{\rm NCIC}_2 = N \hat{d}_{{\rm NCE}} (\hat{\xi}_{{\rm NCE}}) + m -\frac{1}{N+M} \left\{ \sum_{t=1}^N \hat{b}(x^{(t)}) + \sum_{t=1}^M \hat{b}(y^{(t)}) \right\} \label{NCIC}
	\end{align}
	as an approximately unbiased estimator of $N {\rm E}_{x,y} \{ d_{{\rm NCE}} (q,\hat{p}) \}$.
	
	By minimizing NCIC, we can select from non-normalized models \eqref{NCEparam} estimated by NCE.
	${\rm NCIC}_1$ \eqref{NCIC0} and ${\rm NCIC}_2$ \eqref{NCIC} are viewed as analogues of TIC \eqref{TIC} and AIC \eqref{AIC} for non-normalized models, respectively. 
	As will be shown in Section~\ref{subsec:bias_simulation}, ${\rm NCIC}_2$ has much smaller variance than ${\rm NCIC}_1$.
	Also, ${\rm NCIC}_2$ is computationally more efficient than ${\rm NCIC}_1$.
	Therefore, ${\rm NCIC}_2$ is recommended to use when the model is considered to be not badly mis-specified.
	This situation is quite similar to that of TIC and AIC \citep[see][Section 2.3]{Burnham}.
	
	Since NCE is an M-estimator, we can also develop Generalized Information Criterion \citep[GIC;][]{Konishi96} for NCE in principle.
	However, GIC involves the log-likelihood of the model and thus requires to compute the intractable normalization constant.
	On the other hand, NCIC can be readily calculated from the result of NCE.
	
	Instead of NCIC, we can also use leave-one-out cross-validation (LOOCV) with NCE for model selection.
	Specifically\footnote{Here, we assume $M=N$ for convenience.}, for $t=1,\dots,N$, let $\hat{\xi}^{(-t)}$ be the estimate of $\xi$ by NCE applied to $x^{(1)},\dots,x^{(t-1)},x^{(t+1)},\dots,x^{(N)}$ and $y^{(1)},\dots,y^{(t-1)},y^{(t+1)},\dots,y^{(N)}$.
	Then, the quantity
	\begin{align}
		\text{NCE-CV} = \sum_{t=1}^N \rho_d(x^{(t)},\hat{\xi}^{(-t)}) + \sum_{t=1}^M \rho_n (y^{(t)},\hat{\xi}^{(-t)}). \label{NCECV}
	\end{align}
	can be adopted as an approximately unbiased estimator of $N {\rm E}_{x,y} \{ d_{{\rm NCE}} (q,\hat{p}) \}$.
	We confirmed by simulation that the model selection performances of NCIC and NCE-CV are comparable, whereas NCIC is computationally more efficient than NCE-CV (Section~\ref{subsec:tGGM}).
	
	In developing NCIC, we assumed that the noise samples are independent.
	Recently, \cite{Chopin} established the asymptotic theory of NCE including cases where the noise samples are generated by MCMC.
	It is an interesting future work to extend NCIC to such general cases.
	Further problems for future research  include extension to generalized NCE \citep{Uehara20a} and missing data \citep{Uehara20b}.
	
	\section{Information criteria for score matching (SMIC)}
	In this section, we develop new information criteria for score matching, which we call the Score Matching Information Criterion (SMIC).
	For convenience, we focus on the original score matching estimator $\hat{\theta}_{{\rm SM}}$ in the following.
	Analogous results for the score matching estimator $\hat{\theta}_{{\rm SM+}}$ for non-negative data are obtained by replacing $\hat{d}_{{\rm SM}}$ and $\rho_{{\rm SM}}$ with $\hat{d}_{{\rm SM+}}$ and $\rho_{{\rm SM+}}$, respectively.
	
	Suppose we have $N$ i.i.d.~samples $x^{(1)},\dots,x^{(N)}$ from an unknown distribution $q(x)$ and fit a non-normalized model \eqref{nnp} with $\theta \in \mathbb{R}^k$ by score matching.
	Here, the true distribution $q(x)$ may not be contained in the assumed non-normalized model.
	We define $\theta^* = {\rm arg} \min_{\theta} d_{{\rm SM}} (q,p_{\theta})$ and write $p_*(x)=p (x \mid \theta^*)$ and $\hat{p}(x)=p(x \mid \hat{\theta}_{{\rm SM}})$.
	Note that $p_*(x)=q(x)$ when the model includes the true distribution.
	
	Define $k \times k$ matrices ${I}(\theta)$ and ${J}(\theta)$ by
	\begin{align*}
		I(\theta) = {\rm Cov}_q \left\{ \nabla_{\theta} \rho_{{\rm SM}} (x,\theta) \right\}, \quad J(\theta) = {\rm E}_q \left\{ \nabla_{\theta}^2 \rho_{{\rm SM}} (x,\theta) \right\}.
	\end{align*}
	Assume the following regularity conditions:
	
	\begin{itemize}
		\item[(S1)] For every $x$, $\log p(x \mid \theta)$ is $C^3$ with respect to $\theta$.
		
		\item[(S2)] ${\rm  E}_q \left[ \nabla_{\theta} \rho_{\mathrm{SM}}(x,\theta^*) \nabla_{\theta} \rho_{\mathrm{SM}}(x,\theta^*)^{\top} \right]$ is finite.
		
		\item[(S3)] There exists a function $b(x)$ such that
		\begin{align*}
			\left| \rho_{\mathrm{SM}}(x,\theta) \right| \leq b(x), \quad \left| \frac{\partial^2}{\partial \theta_i \partial \theta_j} \rho_{\mathrm{SM}}(x,\theta) \right| \leq b(x), \quad \left| \frac{\partial^3}{\partial \theta_i \partial \theta_j \partial \theta_k} \rho_{\mathrm{SM}}(x,\theta) \right| \leq b(x),
		\end{align*}
		for all $i,j,k$ and $\theta$, where ${\rm E}_q [b(x)] < \infty$.
		
		\item[(S4)] The matrix $J(\theta^*)$ is nonsingular.	
	\end{itemize}
	
	The quantity $\hat{d}_{{\rm SM}} (\hat{\theta}_{{\rm SM}})$ has negative bias as an estimator of ${\rm E}_x [d_{{\rm SM}}(q,\hat{p})]$ and we evaluate this bias following Section~\ref{subsec:bias_NCE}.
	Let $D_1 = \hat{d}_{{\rm SM}}(\hat{\theta}_{{\rm SM}}) - \hat{d}_{{\rm SM}}({\theta}^*)$, $D_2 = \hat{d}_{{\rm SM}}({\theta}^*) - {d}_{{\rm SM}}(q,p_{*})$ and $D_3 = {d}_{{\rm SM}}(q,p_*) - {d}_{{\rm SM}}(q,\hat{p})$.
	Then,
	\begin{align*}
		\hat{d}_{{\rm SM}}(\hat{\theta}_{{\rm SM}}) - d_{{\rm SM}}(q,\hat{p}) = D_1+D_2+D_3.
	\end{align*}
	By using a similar argument to Lemma~\ref{lem_bias}, we obtain the following.
	
	\begin{Lemma}\label{lem_bias_SM}
		Under (S1)-(S4),
		\begin{itemize}
			\item[(a)] $N D_1 \stackrel{d}{\longrightarrow} -\frac{1}{2} s^{\top} J(\xi^*) s$, where $s \sim {\rm N} \left\{ 0,J(\xi^*)^{-1} I(\xi^*) J(\xi^*)^{-1} \right\}$.
			
			\item[(b)] ${\rm E}_{q} (D_2)=0$.
			
			\item[(c)] $N D_3 \stackrel{d}{\longrightarrow} -\frac{1}{2} s^{\top} J(\xi^*) s$, where $s \sim {\rm N} \left\{ 0,J(\xi^*)^{-1} I(\xi^*) J(\xi^*)^{-1} \right\}$.
		\end{itemize}
	\end{Lemma}
	
	The expectation of the limit distribution of $N D_1$ and $N D_3$ is
	\begin{align*}
		-\frac{1}{2} {\rm E} [s^{\top} J(\theta^*) s] = -\frac{1}{2} {\rm tr} \left\{ I(\theta^*) J(\theta^*)^{-1} \right\}.
	\end{align*}
	
	Let
	\begin{align*}
		\hat{I} &= \left. \frac{1}{N} \sum_{t=1}^N \nabla_{\theta} \rho_{{\rm SM}} (x^{(t)},\theta) \nabla_{\theta} \rho_{{\rm SM}} (x^{(t)},\theta)^{\top} \right|_{\theta=\hat{\theta}}, \quad \hat{J} = \left. \frac{1}{N} \sum_{t=1}^N \nabla_{\theta}^2 \rho_{{\rm SM}} (x^{(t)},\theta) \right|_{\theta=\hat{\theta}}. 
	\end{align*}
	Then, $\hat{I}$ and $\hat{J}$ are consistent estimators of $I(\theta^*)$ and $J(\theta^*)$, respectively.
	Thus, from Lemma~\ref{lem_bias_SM}, we propose the quantity
	\begin{align}
		{\rm SMIC} = N \hat{d}_{{\rm SM}} (\hat{\theta}_{{\rm SM}}) + {\rm tr} (\hat{I} \hat{J}^{-1}) \label{SMIC}
	\end{align}
	as an approximately unbiased estimator of $N {\rm E}_q \{ d_{{\rm SM}} (q,\hat{p}) \}$.

	For exponential families, the function $\rho_{{\rm SM}}(x,\theta)$ is given by the quadratic form \eqref{exp_SM} and so $\hat{I}$ and $\hat{J}$ in \eqref{SMIC} become simple:
	\begin{align*}
		\hat{I} = \frac{1}{N} \sum_{t=1}^N \left\{ \Gamma(x^{(t)}) \hat{\theta} + g(x^{(t)}) \right\} \left\{ \Gamma(x^{(t)}) \hat{\theta} + g(x^{(t)}) \right\}^{\top}, \quad \hat{J} = \frac{1}{N} \sum_{t=1}^N \Gamma(x^{(t)}). 
	\end{align*}
	On the other hand, for general models, the term $\hat{J}$ in SMIC involves fourth-order partial derivatives.
	Thus, the analytical complexity of SMIC can be larger than NCIC in general.
	Also note that SMIC can be used only for continuous models whereas NCIC is applicable to both continuous and discrete models.
	
	Unlike Lemma~\ref{lem_bias2} for NCIC, it seems difficult to simplify SMIC in the well-specified case due to the derivative with respect to $x$ in the objective functions of score matching.
	Also, note that our focus here is different from \cite{Dawid} and \cite{Shao}, who applied the idea of score matching to Bayesian model selection with improper priors.
	Finally, whereas we can also develop Generalized Information Criterion \citep[GIC;][]{Konishi96} for score matching in principle, GIC is based on the log-likelihood of the model and thus requires to compute the intractable normalization constant, which is not necessary in SMIC.
	
	Instead of SMIC, we can, again, use leave-one-out cross-validation (LOOCV) with score matching for model selection.
	Specifically, for $t=1,\dots,N$, let $\hat{\theta}^{(-t)}$ be the estimate of $\theta$ by score matching applied to $x^{(1)},\dots,x^{(t-1)},x^{(t+1)},\dots,x^{(N)}$.
	Then, the quantity
	\begin{align}
		\text{SM-CV} = \sum_{t=1}^N \rho_{{\rm SM}}(x^{(t)},\hat{\theta}^{(-t)}). \label{SMCV}
	\end{align}
	can be adopted as an approximately unbiased estimator of $N {\rm E}_{q} \{ d_{{\rm SM}} (q,\hat{p}) \}$.
	We confirmed by simulation that the model selection performances of SMIC and SM-CV are comparable, whereas SMIC is computationally more efficient than SM-CV (Section~\ref{subsec:tGGM}) similarly to NCIC.
	
	Recently, \cite{Liu} developed an estimation method for non-normalized models called the Discriminative Likelihood Estimation (DLE), which approximates the Kullback--Leiber divergence by using the techniques of density ratio estimation and Stein operators.
	They also proposed an information criterion based on DLE.
	It is an interesting future work to investigate the relationship of their method with NCIC and SMIC.
	
	\section{Simulation results}
	In this section, we confirm the validity of the proposed information criteria (${\rm NCIC}_1$, ${\rm NCIC}_2$, and SMIC) by simulation.
	For numerical optimization in NCE and score matching, we use the nonlinear conjugate gradient method \citep{Rasmussen}.
	
	\subsection{Accuracy of bias correction}\label{subsec:bias_simulation}
	
	First, we check the accuracy of the bias correction terms in ${\rm NCIC}$ and ${\rm SMIC}$.
	
	\subsubsection{NCIC}
	We generated $N=10^3$ independent samples from the two-component Gaussian mixture distribution $(1-\varepsilon) \cdot {\rm N} (0,1) + \varepsilon \cdot {\rm N} (0,10)$, where $\varepsilon$ specifies the proportion of outliers.
	Then, we applied NCE to estimate the parameters of the non-normalized model
	\begin{align}
		p(x \mid \theta) \propto \exp (\theta_{1} x^2 + \theta_{2} x), \label{ngaussian}
\end{align}
	which is a non-normalized version of the Gaussian distribution ($m=3$).
	The $M=10^3$ noise samples were generated from ${\rm N} (0,1)$ independently.
	When $\varepsilon=0$, the true distribution is included in the model \eqref{ngaussian}.
	This experimental setting follows \cite{Konishi96}.
	
	In Section 5, ${\rm NCIC}_1$ and ${\rm NCIC}_2$ were developed by correcting the bias of the quantity $N \hat{d}_{{\rm NCE}} (\hat{\xi}_{{\rm NCE}})$ as an estimator of $N {\rm E}_{x,y} [d_{{\rm NCE}}(q,\hat{p})]$.
	Namely, the true bias is
	\begin{align*}
		B= N {\rm E}_{x,y} \left\{ \hat{d}_{{\rm NCE}} (\hat{\xi}_{{\rm NCE}}) \right\} - N {\rm E}_{x,y} \left\{ d_{{\rm NCE}}(q,\hat{p}) \right\},
	\end{align*}
	and ${\rm NCIC}_1$ in \eqref{NCIC0} and ${\rm NCIC}_2$ in \eqref{NCIC} are based on the bias estimates
	\begin{align*}
		\hat{B}_1 = -{\rm tr} (\hat{I} \hat{J}^{-1}), \end{align*}
	and
	\begin{align*}
		\hat{B}_2 = -m+\frac{1}{N+M} \left\{ \sum_{t=1}^N \hat{b}(x^{(t)}) + \sum_{t=1}^M \hat{b}(y^{(t)}) \right\},
	\end{align*}
	respectively.
	We compare these values numerically by a Monte Carlo simulation with $10^5$ repetitions.
	
	Figure~\ref{fig_bias} plots $B$, ${\rm E}_{x,y} (\hat{B}_1)$ and ${\rm E}_{x,y} (\hat{B}_2)$ as a function of $\varepsilon$.
	When $\varepsilon=0$ (well-specified case), the bias $B$ is approximately equal to $-(m-1)=-2$ and both ${\rm E}_{x,y} (\hat{B}_1)$ and ${\rm E}_{x,y} (\hat{B}_2)$ are close to this value.
	When $\varepsilon>0$ (mis-specified case), $B$ and ${\rm E}_{x,y} (\hat{B}_1)$ coincide quite well.
	These results are consistent with Lemma~\ref{lem_bias} and \ref{lem_bias2}.
	Whereas the standard deviation of $\hat{B}_1$ is around 0.1 (see dotted lines in Figure~\ref{fig_bias}), that of $\hat{B}_2$ is smaller than $10^{-8}$.
	Thus, ${\rm NCIC}_2$ has much smaller variance than ${\rm NCIC}_1$.
	This is analogous to the fact that TIC has much larger variance than AIC \citep{Burnham}.
	Interestingly, the absolute bias $|B|$ decreases with $\varepsilon$, whereas it increases with $\varepsilon$ for normalized models (see Fig.~1 of Konishi and Kitagawa, 1996). 
	Thus, just using the number of parameters as the bias correction term may be fairly useful and robust in practice, just like AIC is attractive in that the bias correction term is the  number of parameters.
	Therefore, ${\rm NCIC}_2$ is recommended to use when the model is considered to be not badly mis-specified.
	This situation is quite similar to that of TIC and AIC \citep[see][Section 2.3]{Burnham}.
	
	\begin{figure}
		\begin{center}
			\includegraphics[width=8cm]{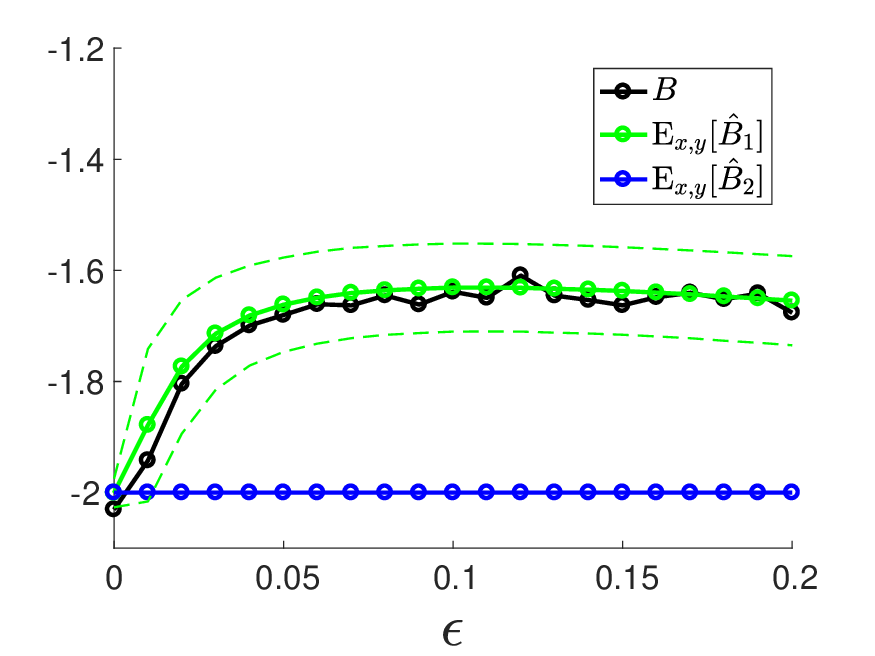}
		\end{center}
		\caption{Comparison of the true bias $B$ (black) and the bias estimates $\hat{B}_1$ (green, with standard deviation) and $\hat{B}_2$ (blue, standard deviation $< 10^{-8}$) in NCIC. Here, $\varepsilon=0$ means a well-specified model.}
		\label{fig_bias}
	\end{figure}
	
	\subsubsection{SMIC}
	We generated $N=10^3$ independent samples from the two-component Gaussian mixture distribution $(1-\varepsilon) \cdot {\rm N} (0,1) + \varepsilon \cdot {\rm N} (0,10)$.
	Then, we applied score matching to fit the normal distribution \eqref{ngaussian}.
	When $\varepsilon=0$, the true distribution is included in the model \eqref{ngaussian}.
	This experimental setting follows \cite{Konishi96}.
	In this case, the model is exponential family and the functions in \eqref{exp_SM} is
	\begin{align*}
		\Gamma(x) = \frac{1}{N} \sum_{t=1}^N \begin{pmatrix} 8 x_t^2 & 4 x_t \\ 4 x_t & 2 \end{pmatrix}, \quad g(x) = \begin{pmatrix} 4 \\ 0 \end{pmatrix}, \quad c(x)=0.
	\end{align*}
	
	In SMIC, the true bias $B= N {\rm E}_q \{ \hat{d}_{{\rm SM}} (\hat{\theta}_{{\rm SM}}) \} - N {\rm E}_q \{ d_{{\rm SM}}(q,\hat{p}) \}$ is estimated by $\hat{B} = -{\rm tr} ( \hat{I} \hat{J}^{-1} )$.
	Figure~\ref{fig_bias_SMIC} plots $B$ and ${\rm E}_q (\hat{B})$ as a function of $\varepsilon$.
	These values were computed by a Monte Carlo simulation with $10^5$ repetitions.
	Consistent with Lemma~\ref{lem_bias_SM}, $B$ and ${\rm E}_q (\hat{B})$ coincide quite well.
	Note that the bias goes down before going up again as $\epsilon$ increases.
	
	\begin{figure}
		\begin{center}
			\includegraphics[width=8cm]{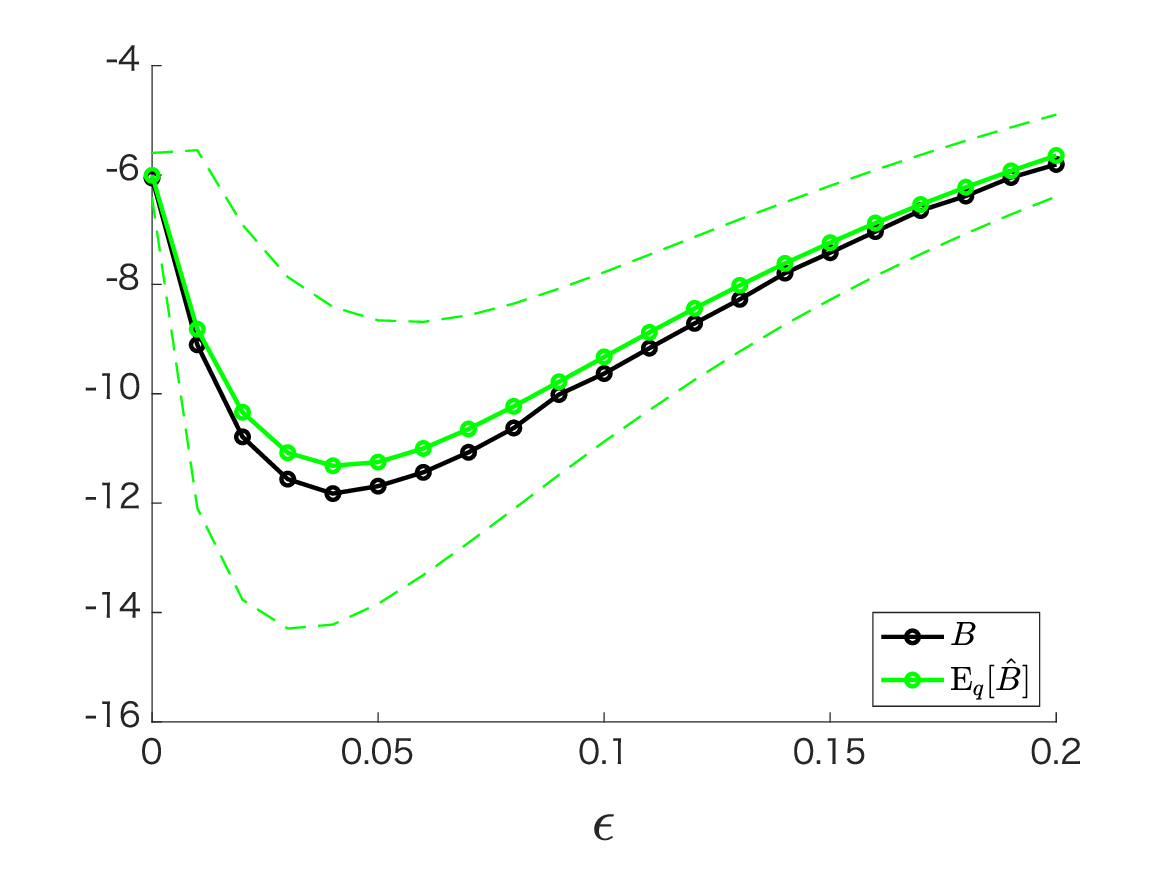}
		\end{center}
		\caption{Comparison of the true bias $B$ (black) and the bias estimate $\hat{B}$ (green, with standard deviation) in SMIC.}
		\label{fig_bias_SMIC}
	\end{figure}
	
	\subsection{Gaussian graphical model}
	
	Next, we apply ${\rm NCIC}$ and ${\rm SMIC}$ to edge selection of the Gaussian graphical model (GGM) \citep{Lauritzen} and compare their performance with AIC.
	
	Let $G=(V,E)$ be an undirected graph where $V=\{ 1,\dots,d \}$. 
	Then, the GGM with graph $G$ is defined as
	\begin{align}
		p(x \mid \Sigma) = \frac{1}{(2 \pi)^{d/2} (\det \Sigma)^{1/2}} \exp \left( -\frac{1}{2} x^{\top} \Sigma^{-1} x \right), \quad x \in \mathbb{R}^d, \label{GGM}
	\end{align}
	where $\Sigma \in \mathbb{R}^{d \times d}$ is a positive definite matrix satisfying $(\Sigma^{-1})_{ij}=0$ for $(i,j) \not\in E$ and the normalization constant is obtained in closed form.
	The zero-nonzero pattern of the precision matrix $\Sigma^{-1}$ specifies the conditional independence structure of $X=(X_1,\cdots,X_d)$: if $(\Sigma^{-1})_{ij}=0 \ (i \neq j)$, then $X_i$ and $X_j$ are  independent conditionally on the other variables $X_k \ (k \neq i,j)$.
	Thus, we consider selection of the graph $G$.

	Following \cite{Drton}, we generated $N$ independent samples $x^{(1)},\dots,x^{(N)}$ from ${\rm N}(0,\Sigma)$ with
	\begin{align}
		\Sigma^{-1} = \begin{pmatrix} 1 & \sigma^{12} & 0 \\ \sigma^{12} & 1 & 0.55 \\ 0 & 0.55 & 1 \end{pmatrix}, \label{Drton_SIGMA}
	\end{align}
	where the value of $\sigma^{12}$ is set to $0.2$, $0.3$ or $0.5$.
	This distribution corresponds to the GGM \eqref{GGM} with the path graph of size $d=3$: $G_3=(V_3,E_3)$ where $V_3 = \{1,2,3\}$ and $E_3=\{(1,2),(2,3)\}$.
	Then, we fitted $2^{d(d-1)/2}=8$ GGMs \eqref{GGM} with each possible $G$ to $x^{(1)},\dots,x^{(N)}$ by using NCE, score matching or maximum likelihood estimation (MLE).
	Namely, we estimated both diagonal and off-diagonal elements of $\Sigma^{-1}$ under the constraint $(\Sigma^{-1})_{ij}=0$ for $(i,j) \not\in E$.
	For NCE, we generated $M=N$ noise samples $y^{(1)},\dots,y^{(M)}$ from the normal distribution with the same mean and covariance with $x^{(1)},\dots,x^{(N)}$.
	For MLE, we used CVX, a MATLAB package for convex programming \citep{cvx}.
	We selected $G$ that corresponds to the GGM with the minimum ${\rm NCIC}_1$, ${\rm NCIC}_2$, ${\rm SMIC}$ or ${\rm AIC}$.
	We repeated the simulation 1000 times.
	
	Tables~\ref{tab_GGM1}--\ref{tab_GGM3} present the detection probabilities of each edge for $N=100$, $N=200$ and $N=1000$, respectively.
	For all criteria, the edges in $G_3$, namely $(1,2)$ and $(2,3)$, are selected more frequently than the edge absent in $G_3$, namely $(1,3)$, especially when $N$ is large.
	Furthermore, the frequency of selecting the edge $(1,2)$ increases with the magnitude of $\sigma^{12}$.
	SMIC attains almost the same performance with AIC, which is consistent with the fact that the score matching estimator coincides with the MLE for Gaussian models \citep{SM}.
	On the other hand, the performance of NCIC is a little worse than SMIC and AIC, which is reasonable because NCE has larger asymptotic variance than MLE \citep{Uehara}.
	
	\begin{table}
		\caption{Detection probabilities of each edge in the GGM \eqref{GGM} for $N=100$ when (a) $\sigma^{12}=0.2$, (b) $\sigma^{12}=0.3$ and (c) $\sigma^{12}=0.5$. The true edges are $(1,2)$ and $(2,3)$.}
		\label{tab_GGM1}
		\centering	
		\begin{minipage}{0.5\textwidth}
			(a)\\
			\begin{tabular}{|r|c|c|c|c|} \hline
				& NCIC${}_1$ & NCIC${}_2$ & SMIC & AIC  \\ \hline
				(1,2) & 0.515 & 0.481 & 0.790 & 0.783 \\
				(1,3) & 0.187 & 0.170 & 0.199 & 0.210 \\
				(2,3) & 0.945 & 0.928 & 1.000 & 1.000 \\ \hline
			\end{tabular}
		\end{minipage}\\
		\begin{minipage}{0.5\textwidth}
			(b)\\
			\begin{tabular}{|r|c|c|c|c|} \hline
				& NCIC${}_1$ & NCIC${}_2$ & SMIC & AIC  \\ \hline
				(1,2) & 0.750 & 0.706 & 0.966 & 0.971 \\
				(1,3) & 0.198 & 0.181 & 0.167 & 0.165 \\
				(2,3) & 0.947 & 0.930 & 1.000 & 1.000 \\ \hline
			\end{tabular}
		\end{minipage}\\
		\begin{minipage}{0.5\textwidth}
			(c)\\
			\begin{tabular}{|r|c|c|c|c|} \hline
				& NCIC${}_1$ & NCIC${}_2$ & SMIC & AIC  \\ \hline
				(1,2) & 0.943 & 0.926 & 1.000 & 1.000 \\
				(1,3) & 0.190 & 0.171 & 0.145 & 0.145 \\
				(2,3) & 0.953 & 0.941 & 1.000 & 1.000 \\ \hline
			\end{tabular}
		\end{minipage}
	\end{table}
	
	\begin{table}
		\caption{Detection probabilities of each edge in the GGM \eqref{GGM} for $N=200$ when (a) $\sigma^{12}=0.2$, (b) $\sigma^{12}=0.3$ and (c) $\sigma^{12}=0.5$. The true edges are $(1,2)$ and $(2,3)$.}
		\label{tab_GGM2}
		\centering	
		\begin{minipage}{0.5\textwidth}
			(a)\\
			\begin{tabular}{|r|c|c|c|c|} \hline
				& NCIC${}_1$ & NCIC${}_2$ & SMIC & AIC  \\ \hline
				(1,2) & 0.749 & 0.719 & 0.938 & 0.936 \\ 
				(1,3) & 0.218 & 0.210 & 0.167 & 0.170 \\ 
				(2,3) & 1.000 & 0.999 & 1.000 & 1.000 \\ \hline 
			\end{tabular}
		\end{minipage}\\
		\begin{minipage}{0.5\textwidth}
			(b)\\
			\begin{tabular}{|r|c|c|c|c|} \hline
				& NCIC${}_1$ & NCIC${}_2$ & SMIC & AIC  \\ \hline
				(1,2) & 0.937 & 0.925 & 1.000 & 0.999 \\ 
				(1,3) & 0.172 & 0.162 & 0.123 & 0.137 \\ 
				(2,3) & 1.000 & 1.000 & 1.000 & 1.000 \\ \hline 
			\end{tabular}
		\end{minipage}\\
		\begin{minipage}{0.5\textwidth}
			(c)\\
			\begin{tabular}{|r|c|c|c|c|} \hline
				& NCIC${}_1$ & NCIC${}_2$ & SMIC & AIC  \\ \hline
				(1,2) & 0.997 & 0.997 & 1.000 & 1.000 \\ 
				(1,3) & 0.147 & 0.138 & 0.147 & 0.139 \\ 
				(2,3) & 0.997 & 0.997 & 1.000 & 1.000 \\ \hline 
			\end{tabular}
		\end{minipage}
	\end{table}
	
	\begin{table}
		\caption{Detection probabilities of each edge in the GGM \eqref{GGM} for $N=1000$ when (a) $\sigma^{12}=0.2$, (b) $\sigma^{12}=0.3$ and (c) $\sigma^{12}=0.5$. The true edges are $(1,2)$ and $(2,3)$.}
		\label{tab_GGM3}
		\centering	
		\begin{minipage}{0.5\textwidth}
			(a)\\
			\begin{tabular}{|r|c|c|c|c|} \hline
				& NCIC${}_1$ & NCIC${}_2$ & SMIC & AIC  \\ \hline
				(1,2) & 0.999 & 0.999 & 1.000 & 1.000 \\ 
				(1,3) & 0.167 & 0.162 & 0.145 & 0.143 \\ 
				(2,3) & 1.000 & 1.000 & 1.000 & 1.000 \\ \hline 
			\end{tabular}
		\end{minipage}\\
		\begin{minipage}{0.5\textwidth}
			(b)\\
			\begin{tabular}{|r|c|c|c|c|} \hline
				& NCIC${}_1$ & NCIC${}_2$ & SMIC & AIC  \\ \hline
				(1,2) & 1.000 & 1.000 & 1.000 & 1.000 \\ 
				(1,3) & 0.169 & 0.166 & 0.155 & 0.159 \\ 
				(2,3) & 1.000 & 1.000 & 1.000 & 1.000 \\ \hline 
			\end{tabular}
		\end{minipage}\\
		\begin{minipage}{0.5\textwidth}
			(c)\\
			\begin{tabular}{|r|c|c|c|c|} \hline
				& NCIC${}_1$ & NCIC${}_2$ & SMIC & AIC  \\ \hline
				(1,2) & 1.000 & 1.000 & 1.000 & 1.000 \\ 
				(1,3) & 0.150 & 0.148 & 0.147 & 0.142 \\ 
				(2,3) & 1.000 & 1.000 & 1.000 & 1.000 \\ \hline 
			\end{tabular}
		\end{minipage}
	\end{table}
	
	\subsection{Truncated Gaussian graphical model}\label{subsec:tGGM}
	Now, we apply ${\rm NCIC}$ and ${\rm SMIC}$ to edge selection of the truncated Gaussian graphical model \citep{Lin}, which has an intractable normalization constant.
	
	For an undirected graph $G=(V,E)$ with $V=\{ 1,\dots,d \}$, the truncated GGM with graph $G$ is defined as
	\begin{align}
		p(x \mid \Sigma) \propto \exp \left( -\frac{1}{2} x^{\top} \Sigma^{-1} x \right), \quad x \in \mathbb{R}_+^d, \label{tGGM}
	\end{align}
	where $\Sigma \in \mathbb{R}^{d \times d}$ is a positive definite matrix satisfying $(\Sigma^{-1})_{ij}=0$ for $(i,j) \not\in E$.
	Due to the truncation to the positive orthant $\mathbb{R}_+^d$, the normalization constant of the truncated GGM \eqref{tGGM} is computationally intractable.
	Similarly to the original GGM \eqref{GGM}, $X_i$ and $X_j$ are  independent conditionally on the other variables $X_k \ (k \neq i,j)$ if $(\Sigma^{-1})_{ij}=0$.
	Thus, we consider selection of the graph $G$.
	
	Similarly to the previous subsection, we considered edge selection from $N$ independent samples $x^{(1)},\dots,x^{(N)}$ from a truncated GGM \eqref{tGGM} with covariance \eqref{Drton_SIGMA} where $\sigma^{12}$ is set to $0.2$, $0.3$ or $0.5$.
	For NCE, we generated $M=N$ noise samples $y^{(1)},\dots,y^{(M)}$ from the product of the coordinate-wise exponential distributions with the same mean as $x^{(1)},\dots,x^{(N)}$.
	We selected $G$ that corresponds to the truncated GGM with the minimum ${\rm NCIC}_1$, ${\rm NCIC}_2$ or ${\rm SMIC}$.
	For comparison with model selection by leave-one-out cross-validation (LOOCV), we also selected $G$ by minimizing NCE-CV \eqref{NCECV} or SM-CV \eqref{SMCV}.
	We repeated the simulation 1000 times.
	
	Tables~\ref{tab_tGGM1}--\ref{tab_tGGM3} present the detection probabilities of each edge for $N=100$, $N=200$ and $N=1000$, respectively.
	The behaviors of NCIC and SMIC are qualitatively the same with Tables~\ref{tab_GGM1}--\ref{tab_GGM3}.
	Namely, the true edges $(1,2)$ and $(2,3)$ are selected more frequently than the false edge $(1,3)$ especially when $N$ is large, and the frequency of selecting the edge $(1,2)$ increases with the magnitude of $\sigma^{12}$.
	Also, the model selection performances of NCIC and SMIC are comparable to those of NCE-CV and SM-CV, respectively, which is analogous to the asymptotic equivalence of model selection by AIC and LOOCV for normalized models \citep{Stone}.
	Note that NCE-CV and SM-CV take approximately $N$ times more computational cost than NCIC and SMIC, respectively.

	\begin{table}
		\caption{Detection probabilities of each edge in the truncated GGM \eqref{tGGM} for $N=100$ when (a) $\sigma^{12}=0.2$, (b) $\sigma^{12}=0.3$ and (c) $\sigma^{12}=0.5$. The true edges are $(1,2)$ and $(2,3)$.}
		\label{tab_tGGM1}
		\centering	
		\begin{minipage}{0.6\textwidth}
			(a)\\
			\begin{tabular}{|r|c|c|c|c|c|} \hline
				& NCIC${}_1$ & NCIC${}_2$ & NCE-CV & SMIC & SM-CV  \\ \hline
				(1,2) & 0.287 & 0.221 & 0.210 & 0.360 & 0.243 \\ 
				(1,3) & 0.187 & 0.155 & 0.132 & 0.303 & 0.184 \\ 
				(2,3) & 0.495 & 0.429 & 0.422 & 0.617 & 0.513 \\ \hline 
			\end{tabular}
		\end{minipage}\\
		\begin{minipage}{0.6\textwidth}
			(b)\\
			\begin{tabular}{|r|c|c|c|c|c|} \hline
				& NCIC${}_1$ & NCIC${}_2$ & NCE-CV & SMIC & SM-CV  \\ \hline
				(1,2) & 0.316 & 0.251 & 0.249 & 0.449 & 0.337 \\ 
				(1,3) & 0.205 & 0.182 & 0.152 & 0.304 & 0.205 \\ 
				(2,3) & 0.522 & 0.451 & 0.449 & 0.603 & 0.506 \\ \hline 
			\end{tabular}
		\end{minipage}\\
		\begin{minipage}{0.6\textwidth}
			(c)\\
			\begin{tabular}{|r|c|c|c|c|c|} \hline
				& NCIC${}_1$ & NCIC${}_2$ & NCE-CV & SMIC & SM-CV  \\ \hline
				(1,2) & 0.460 & 0.405 & 0.394 & 0.592 & 0.478 \\ 
				(1,3) & 0.206 & 0.173 & 0.147 & 0.308 & 0.204 \\ 
				(2,3) & 0.496 & 0.427 & 0.430 & 0.594 & 0.498 \\ \hline 
			\end{tabular}
		\end{minipage}
	\end{table}
	
	\begin{table}
		\caption{Detection probabilities of each edge in the truncated GGM \eqref{tGGM} for $N=200$ when (a) $\sigma^{12}=0.2$, (b) $\sigma^{12}=0.3$ and (c) $\sigma^{12}=0.5$. The true edges are $(1,2)$ and $(2,3)$.}
		\label{tab_tGGM2}
		\centering	
		\begin{minipage}{0.6\textwidth}
			(a)\\
			\begin{tabular}{|r|c|c|c|c|c|} \hline
				& NCIC${}_1$ & NCIC${}_2$ & NCE-CV & SMIC & SM-CV  \\ \hline
				(1,2) & 0.302 & 0.256 & 0.265 & 0.383 & 0.302 \\ 
				(1,3) & 0.195 & 0.178 & 0.156 & 0.278 & 0.198 \\ 
				(2,3) & 0.689 & 0.642 & 0.651 & 0.730 & 0.667 \\ \hline 
			\end{tabular}
		\end{minipage}\\
		\begin{minipage}{0.6\textwidth}
			(b)\\
			\begin{tabular}{|r|c|c|c|c|c|} \hline
				& NCIC${}_1$ & NCIC${}_2$ & NCE-CV & SMIC & SM-CV  \\ \hline
				(1,2) & 0.436 & 0.371 & 0.392 & 0.482 & 0.424 \\ 
				(1,3) & 0.191 & 0.173 & 0.156 & 0.251 & 0.179 \\ 
				(2,3) & 0.704 & 0.659 & 0.670 & 0.744 & 0.682 \\ \hline 
			\end{tabular}
		\end{minipage}\\
		\begin{minipage}{0.6\textwidth}
			(c)\\
			\begin{tabular}{|r|c|c|c|c|c|} \hline
				& NCIC${}_1$ & NCIC${}_2$ & NCE-CV & SMIC & SM-CV  \\ \hline
				(1,2) & 0.647 & 0.585 & 0.613 & 0.713 & 0.646 \\ 
				(1,3) & 0.201 & 0.185 & 0.161 & 0.284 & 0.191 \\ 
				(2,3) & 0.692 & 0.639 & 0.670 & 0.728 & 0.665 \\ \hline 
			\end{tabular}
		\end{minipage}
	\end{table}
	
	\begin{table}
		\caption{Detection probabilities of each edge in the truncated GGM \eqref{tGGM} for $N=1000$ when (a) $\sigma^{12}=0.2$, (b) $\sigma^{12}=0.3$ and (c) $\sigma^{12}=0.5$. The true edges are $(1,2)$ and $(2,3)$.}
		\label{tab_tGGM3}
		\centering	
		\begin{minipage}{0.6\textwidth}
			(a)\\
			\begin{tabular}{|r|c|c|c|c|c|} \hline
				& NCIC${}_1$ & NCIC${}_2$ & NCE-CV & SMIC & SM-CV  \\ \hline
				(1,2) & 0.613 & 0.586 & 0.604 & 0.623 & 0.599 \\ 
				(1,3) & 0.171 & 0.162 & 0.162 & 0.184 & 0.156 \\ 
				(2,3) & 0.996 & 0.996 & 0.996 & 0.993 & 0.992 \\ \hline 
			\end{tabular}
		\end{minipage}\\
		\begin{minipage}{0.6\textwidth}
			(b)\\
			\begin{tabular}{|r|c|c|c|c|c|} \hline
				& NCIC${}_1$ & NCIC${}_2$ & NCE-CV & SMIC & SM-CV  \\ \hline
				(1,2) & 0.839 & 0.820 & 0.830 & 0.829 & 0.809 \\ 
				(1,3) & 0.163 & 0.160 & 0.154 & 0.193 & 0.167 \\ 
				(2,3) & 0.996 & 0.996 & 0.996 & 0.995 & 0.993 \\ \hline 
			\end{tabular}
		\end{minipage}\\
		\begin{minipage}{0.6\textwidth}
			(c)\\
			\begin{tabular}{|r|c|c|c|c|c|} \hline
				& NCIC${}_1$ & NCIC${}_2$ & NCE-CV & SMIC & SM-CV  \\ \hline
				(1,2) & 0.985 & 0.983 & 0.982 & 0.971 & 0.965 \\ 
				(1,3) & 0.183 & 0.172 & 0.173 & 0.195 & 0.170 \\ 
				(2,3) & 0.995 & 0.995 & 0.995 & 0.990 & 0.988 \\ \hline 
			\end{tabular}
		\end{minipage}
	\end{table}
	
	To verify the performance of NCIC and SMIC in higher dimension, we conducted another experiment with $N=1000$, $d=16$ and $G$ given by the grid graph in Figure~\ref{fig_grid}.
	The nonzero off-diagonal entries of $\Sigma^{-1}$ were all set to 0.5 and the diagonal entries of $\Sigma^{-1}$ were set to a common value so that the minimum eigenvalue of $\Sigma^{-1}$ is 0.1.
	Since the number of possible graph structures is too large in this case, we narrowed down the candidate graphs by using a similar procedure to the graphical LASSO.
	Specifically, we first applied the NCE and score matching with $l_1$-regularization on the off-diagonal entries of $\Sigma^{-1}$.
	For optimization, we employed the accelerated proximal gradient algorithm\footnote{We used the MATLAB program from \url{https://github.com/bodono/apg}.}.
	By changing the value of the regularization parameter, a sequence of candidate graphs was obtained for both NCE and score matching.
	Then, we fitted the graphical models for candidate graphs without regularization to calculate NCIC and SMIC.
	Note that such a procedure is also used for LASSO \citep{Belloni}.
	For NCE, we generated $M=N$ noise samples $y^{(1)},\dots,y^{(M)}$ from the product of the coordinate-wise exponential distributions with the same mean as $x^{(1)},\dots,x^{(N)}$.
	We repeated the simulation 1000 times.
	Table~\ref{tab_tGGM4} presents the true positive rate (the probability of selecting the edges in $G$) and false positive rate (the probability of selecting the edges not in $G$).
	It indicates that both NCIC and SMIC select the edges in $G$ much more frequently than those not in $G$.
	The detection performance is not very strong compared to Table~\ref{tab_tGGM3}, which may be related to the difficulty in selecting an appropriate noise distribution in higher dimensions.
	
	\begin{figure}[h]
		\begin{center}
			\begin{tikzpicture}[every node/.style={circle,draw}]
				\node (A) at (0,0) {};
				\node (B) at (1,0) {};
				\node (C) at (2,0) {};
				\node (D) at (3,0) {};
				\node (E) at (0,1) {};
				\node (F) at (1,1) {};
				\node (G) at (2,1) {};
				\node (H) at (3,1) {};
				\node (I) at (0,2) {};
				\node (J) at (1,2) {};
				\node (K) at (2,2) {};
				\node (L) at (3,2) {};
				\node (M) at (0,3) {};
				\node (N) at (1,3) {};
				\node (O) at (2,3) {};
				\node (P) at (3,3) {};
				\foreach \u \v in {A/B,B/C,C/D,E/F,F/G,G/H,I/J,J/K,K/L,M/N,N/O,O/P,A/E,E/I,I/M,B/F,F/J,J/N,C/G,G/K,K/O,D/H,H/L,L/P}
				\draw (\u) -- (\v);
			\end{tikzpicture}
		\end{center}
		\caption{Grid graph.} 
		\label{fig_grid}
	\end{figure}
	
	\begin{table}
		\caption{True and false positive rates for the truncated GGM \eqref{tGGM} with the grid graph.}
		\label{tab_tGGM4}
		\vspace{0.1in}
		\begin{center}	
			\begin{tabular}{|r|c|c|c|c|c|} \hline
				& NCIC${}_1$ & NCIC${}_2$ & SMIC  \\ \hline
				true positive & 0.711 & 0.662 & 0.759 \\ 
				false positive & 0.191 & 0.156 & 0.207 \\ \hline 
			\end{tabular}
		\end{center}	
	\end{table}
	
	\section{Application to real data}
	In this section, we apply NCIC and SMIC to model selection for real data of natural image, RNAseq and wind direction.
	
	\subsection{Natural image data}
	First, we apply NCIC to analysis of natural image data with the energy-based overcomplete independent component analysis (ICA) model \citep{Teh} defined by
	\begin{align}
		\log p(x \mid w) = \sum_{b=1}^B G(w_b^{\top} x) - \log Z(w_1,\dots,w_B), \quad x \in \mathbb{R}^d, \label{overICA}
	\end{align}
	where $w=(w_1,\dots,w_B)$ with $w_i \in \mathbb{R}^d$ is the overcomplete set of filters and $G(u) = -|u|$\footnote{Although this model does not satisfy the smoothness assumption (N2), the non-smoothness here is essentially the same with that in median estimation \citep[Example 5.24]{vV}. The asymptotic variance is still similarly obtained following Theorem 5.23 of \cite{vV} by assuming the smoothness of the expectation of the objective function rather than the objective function itself. We leave the rigorous argument to future work.}.
	This model is related to ICA with overcomplete bases \citep{ICA_book} and extracts useful features of data.
	In previous work, the number of filters $B \ (>d)$ has been selected arbitrarily.
	Here, we determine $B$ by minimizing NCIC. 
	
	We used $N=5 \times 10^4$ image patches of 8 $\times$ 8 pixels taken from natural images.
	This data is provided in Hoyer's \textit{imageica} package.\footnote{\url{http://www.cs.helsinki.fi/patrik.hoyer/}}
	Following \cite{SM}, we removed the DC component and then applied whitening.
	Thus, the data dimension is $d=63$.    
	For NCE, we used $M=5 \times 10^4$ noise samples from the Gaussian distribution with the same mean and covariance as data.
	
	Figure~\ref{fig_overICA} (a) plots ${\rm NCIC}_2$ as a function of $B$.
	${\rm NCIC}_2$ takes minimum at $B=118$.
	Some of the estimated filters $w_1,\dots,w_B$ when $B=118$ are shown in Figure~\ref{fig_overICA} (b).
	Here, the filters are converted back to the original space from the whitened space for visualization.
	Similarly to the result by score matching \citep{SM}, many filters represent localized patterns in image patches \citep{Olshausen,nis_book}.
	Namely, they take (significantly) nonzero value on only limited regions of images.
	Note that the computation of $\hat{I}$ and $\hat{J}$ in ${\rm NCIC}_1$ was computationally intractable in this case due to the large sample size $N$.
	We did not consider SMIC here because the calculation of $\hat{J}$ in SMIC was analytically complex.
	
	\begin{figure}[h]
		\centering
		\begin{minipage}{0.45\textwidth}
			(a)\\
			\includegraphics[width=7cm]{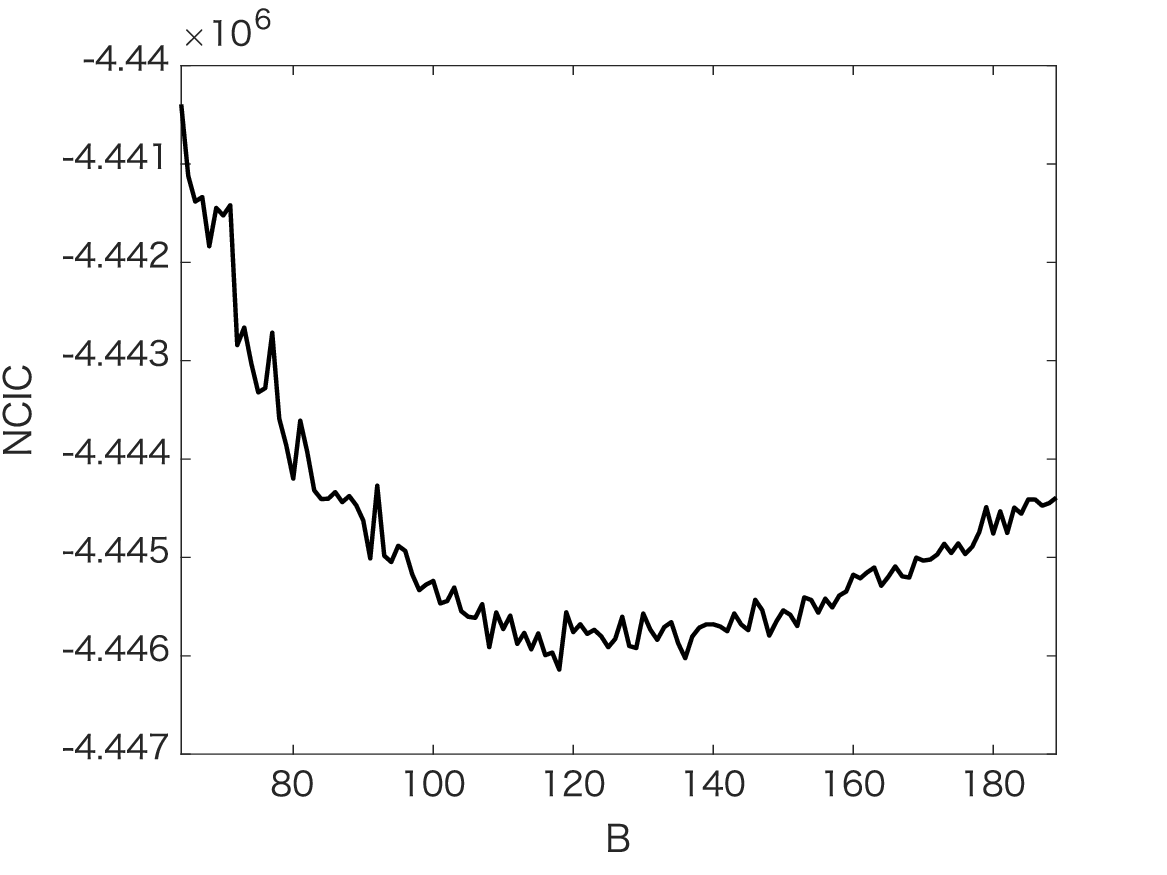}
		\end{minipage}
		\begin{minipage}{0.45\textwidth}
			(b)\\
			\includegraphics[width=7cm]{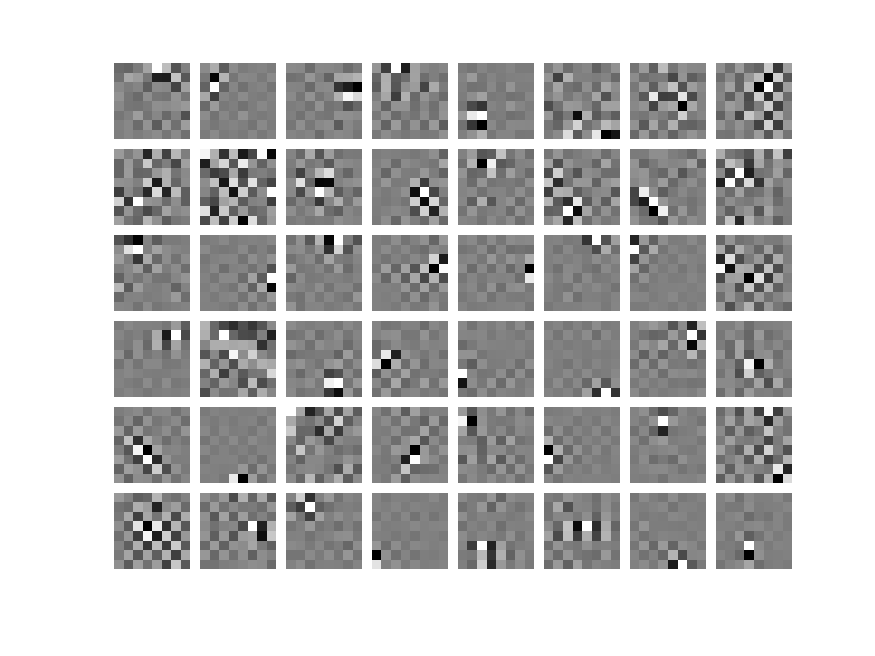}
		\end{minipage}
		\caption{(a) ${\rm NCIC}_2$ of overcomplete ICA models \eqref{overICA} for natural image data. (b) Estimated filters when $B=118$.} 
		\label{fig_overICA}
	\end{figure}
	
	\subsection{RNAseq data}
	Next, we apply SMIC to comparison of graphical model for the RNAseq data used in \cite{Lin}.
	This is a non-negative multivariate data of sample size $N=487$.
	We analyze $d=40$ among 330 genes that do not contain missing values and have coefficient of variation larger than one.
	
	To investigate interaction between genes, \cite{Lin} fitted the truncated Gaussian graphical model \eqref{tGGM} to RNAseq data by $l_1$-regularized score matching, which can be solved by the existing algorithms for LASSO.
	Another possible model is the log-Gaussian graphical model defined by
	\begin{align}
		p(x \mid \mu, \Sigma) \propto \left( \prod_{i=1}^d \frac{1}{x_i} \right) \exp \left( -\frac{1}{2} (\log x-\mu)^{\top} \Sigma^{-1} (\log x-\mu) \right), \quad x \in \mathbb{R}_+^d, \label{lGGM}
	\end{align}
	where $\log$ is applied element-wise.
	Namely, log-transformed data is assumed to follow the usual Gaussian graphical model.
	Note that this model is also an exponential family and thus the objective function of score matching reduces to a quadratic form \eqref{exp_SM}.
	Here, we apply SMIC to determine which of the above two graphical models has better fit to RNAseq data.

	Figure~\ref{fig_RNA} plots SMIC of two graphical models with respect to the number of edges.
	For edge selection, we employed $l_1$ regularized score matching \citep{Lin} for truncated Gaussian graphical models \eqref{tGGM} and graphical LASSO\footnote{We used R package ``glasso" from \url{http://statweb.stanford.edu/~tibs/glasso/}.} for log-Gaussian graphical models \eqref{lGGM}, respectively.
	Namely, we computed the whole regularization paths.
	After edge selection, we fitted the graphical models again by score matching without regularization to calculate SMIC.
	Note that such a procedure is also used for LASSO \citep{Belloni}.
	Figure~\ref{fig_RNA} indicates that SMIC saturates around 400 edges for both models and is smaller for the log-Gaussian graphical model.
	Therefore, the log-Gaussian graphical model has better fit to RNAseq data in this case.
	
	\begin{figure}[h]
		\centering
		\begin{minipage}{0.45\textwidth}
			(a)\\
			\includegraphics[width=7cm]{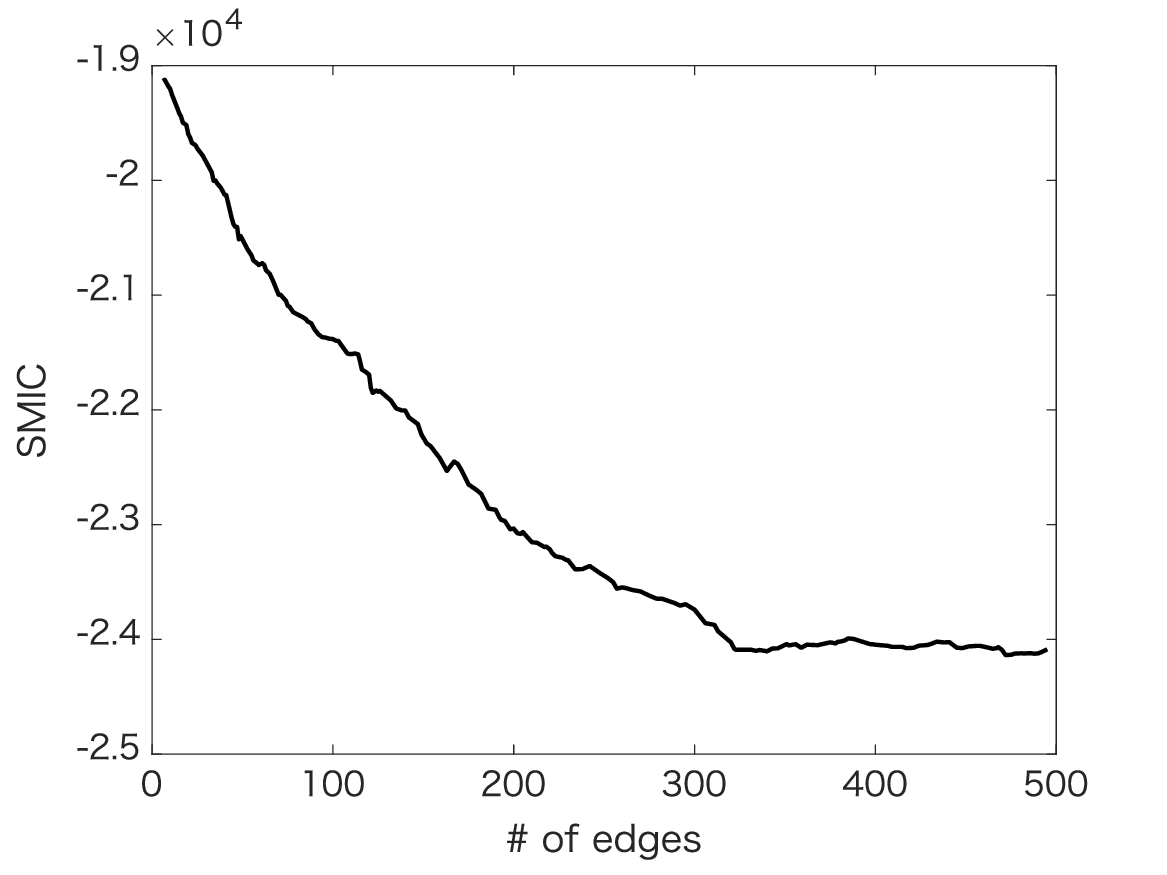}
		\end{minipage}
		\begin{minipage}{0.45\textwidth}
			(b)\\
			\includegraphics[width=7cm]{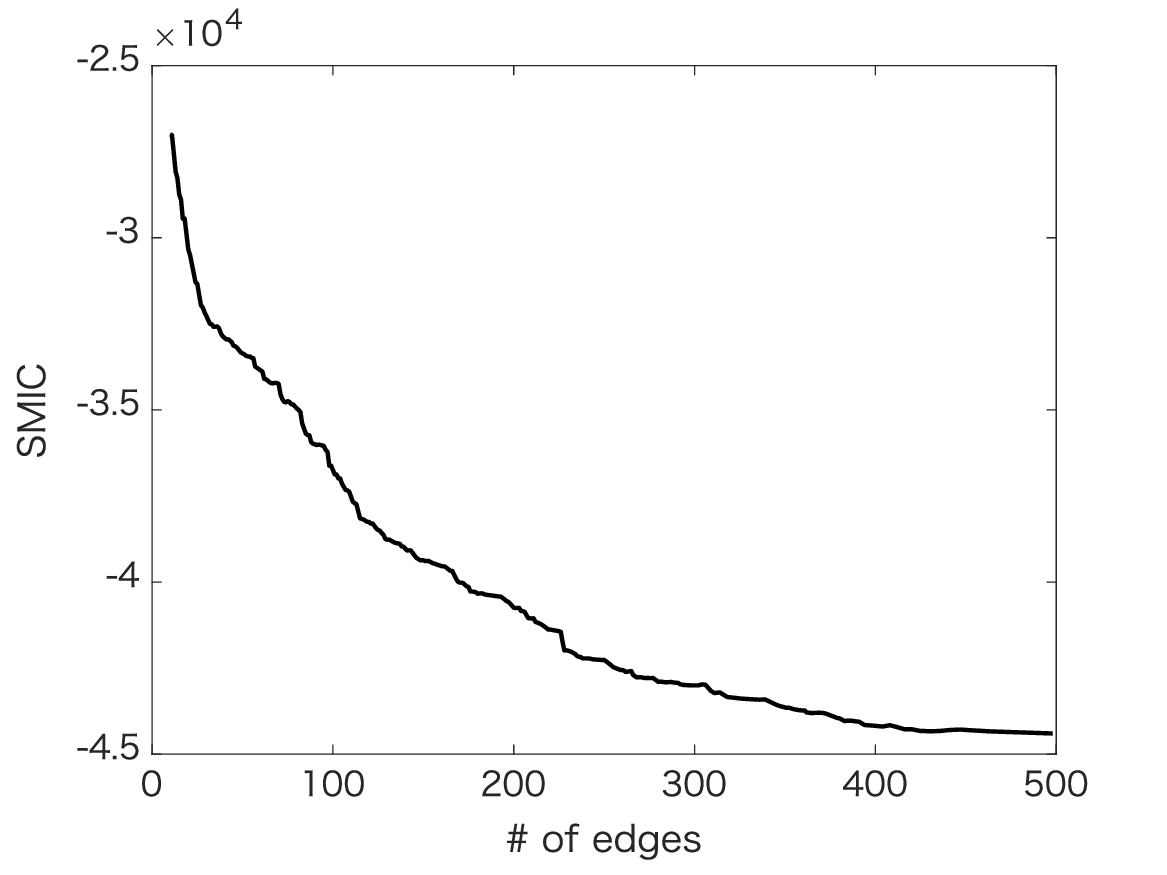}
		\end{minipage}
		\caption{SMIC of (a) truncated Gaussian graphical models \eqref{tGGM} and (b) log-Gaussian graphical models \eqref{lGGM} for RNAseq data. Note that the scale of y-axis is different between (a) and (b).}
		\label{fig_RNA}
	\end{figure}
	
	\subsection{Wind direction data}
	Finally, we apply NCIC to wind direction data.
	Since the wind direction is naturally identified with a vector on the unit circle \citep{Mardia}, we represent it as a circular variable in radians.
	Figure~\ref{fig:wind} shows a 2-d histogram of wind direction at Tokyo on 00:00 ($x_1$) and 12:00 ($x_2$) for $N=365$ days in 2018, which was obtained {from the website of Japan Meteorological Agency}.
	The data are discretized into 16 bins such as north-northeast.
	
	\begin{figure}
		\centering
		\includegraphics[width=8cm]{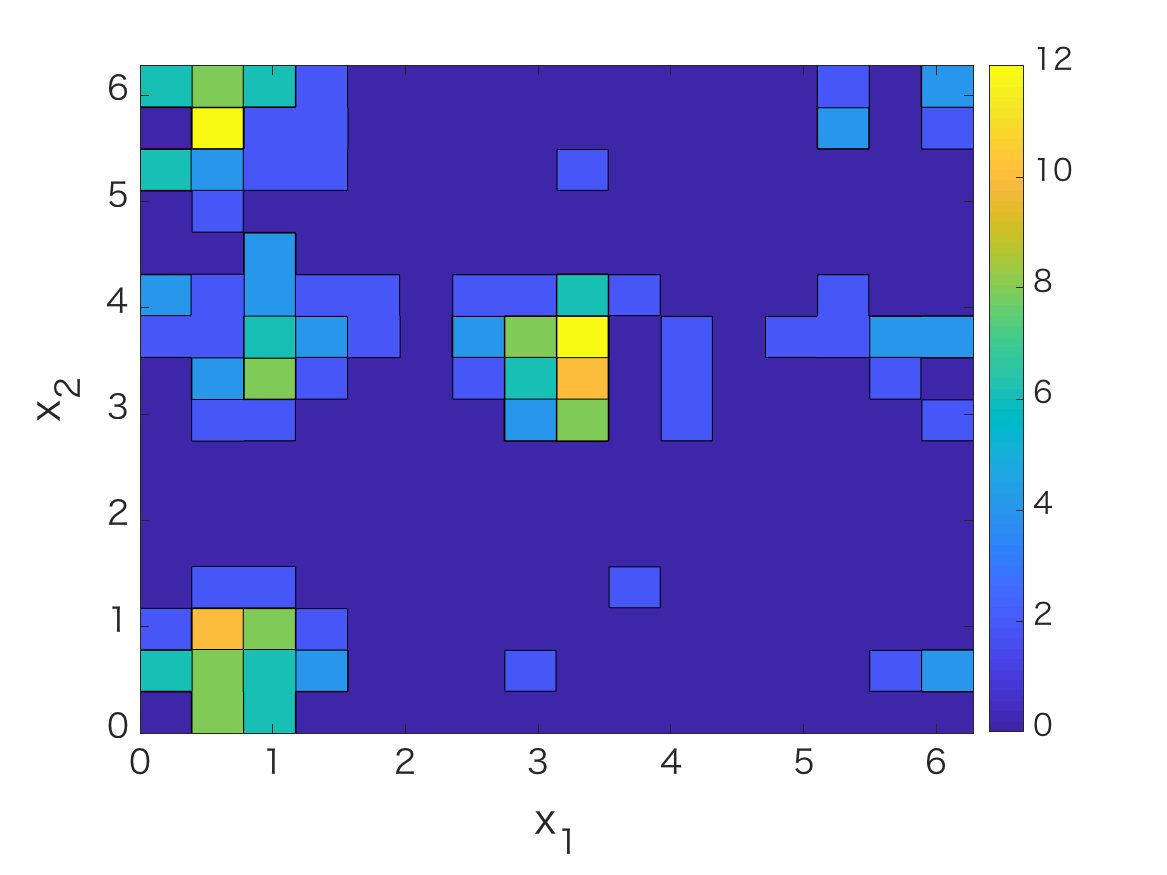} 
		\caption{2-d histogram of wind direction data.}
		\label{fig:wind}
		\vspace{-0.3cm}
	\end{figure}
	
	To describe dependence between two circular variables, \cite{Singh} proposed the bivariate von Mises distribution defined by
	\begin{align}
		{p}(x_1,x_2 \mid \theta) \propto \exp (\kappa_1 \cos(x_1-\mu_1) + \kappa_2 \cos(x_2-\mu_2) + \lambda_{12} \sin(x_1-\mu_1) \sin(x_2-\mu_2)), \label{bvM}
	\end{align}
	where $\theta=(\kappa_1,\kappa_2,\mu_1,\mu_2,\lambda_{12})$ with $\kappa_1 \geq 0$, $\kappa_2 \geq 0$, $0 \leq \mu_1 <2\pi$ and $0 \leq \mu_2 <2\pi$.
	Its normalization constant involves an infinite sum of Bessel functions, which is computationally intractable.
	The parameter $\lambda_{12}$ quantifies the dependency between $x_1$ and $x_2$.
	In particular, $x_1$ and $x_2$ are independent if and only if $\lambda_{12} = 0$.
	
	We fitted the bivariate von Mises distribution \eqref{bvM} to the wind direction data in Figure~\ref{fig:wind} by NCE with $M=1000$ noise samples from the uniform distribution on $[0,2\pi) \times [0,2\pi)$.
	The parameter estimate was
	\begin{align}
		(\hat{\kappa}_1,\hat{\kappa}_2,\hat{\mu}_1,\hat{\mu}_2,\hat{\lambda}_{12})=(0.813,0.440,1.120,4.644,-0.965)
	\end{align}
	and NCIC value was ${\rm NCIC}_2=-1941$.
	We also fitted the bivariate von Mises distribution \eqref{bvM}  with $\lambda_{12} = 0$.
	The parameter estimate was
	\begin{align}
		(\hat{\kappa}_1,\hat{\kappa}_2,\hat{\mu}_1,\hat{\mu}_2)=(0.808,0.430,0.755,4.234)
	\end{align}
	and NCIC value was ${\rm NCIC}_2=-1919$.
	Thus, the former model has better fit than the latter, which implies that the wind direction at Tokyo on 00:00 and 12:00 are dependent.
	
	\section{Extension to non-normalized mixture models}
	In this section, we discuss extension of NCIC to non-normalized mixture models. 
	
	Consider a finite mixture of non-normalized models:
	\begin{align}
		p (x \mid \theta,\pi) &= \sum_{k=1}^K \pi_k \cdot p (x \mid \theta_k), \quad p(x \mid \theta_k) = \frac{1}{Z(\theta_k)} \widetilde{p}(x \mid \theta_k), \label{mix_model}
	\end{align}
	where $\pi_k > 0$, $\sum_{k=1}^K \pi_k=1$ and the normalization constant $Z(\theta_k)$ of each component $p (x \mid \theta_k)$ is intractable.
	Existing methods for estimating non-normalized models are not applicable to \eqref{mix_model} since it includes more than one intractable normalization constant\footnote{For example, a formal extension of score matching becomes intractable because the objective function now involves the normalization constants.}.
	We clarified this point as a footnote.
	Thus, \cite{Matsuda} extended NCE to estimate \eqref{mix_model}.
	Specifically, \eqref{mix_model} is reparametrized as
	\begin{align}
		p (x \mid \theta,c) = \sum_{k=1}^K p(x \mid \theta_k,c_k), \quad 	\log p (x \mid \theta_k,c_k) = \log \widetilde{p} (x \mid \theta_k) + c_k, \label{NCEparam2}
	\end{align}
	where $c=(c_1,\dots,c_K)$ with $c_k = \log \pi_k-\log Z(\theta_k)$.
	Similarly to the original NCE, we consider $c$ as an additional unknown parameter. 
	Then, by generating $M$ noise samples $y^{(1)},\dots,y^{(M)}$ from a noise distribution $n(y)$, the parameter $\xi=(\theta, c)$ is estimated in the same way as the original NCE in \eqref{NCEdef} and \eqref{Jdef},
	that is, we use the definition \eqref{NCEparam2} in the original NCE objective function \eqref{Jdef}.
	This extended NCE has consistency under mild regularity conditions \citep{Matsuda}.
	
	Now, we consider extension of NCIC to non-normalized mixture models.
	The problem setting is essentially the same with Section \ref{sec_NCIC}. 
	Specifically, we have $N$ i.i.d.~samples $x^{(1)},\dots,x^{(N)}$ from an unknown distribution $q(x)$ and estimate a non-normalized mixture model \eqref{NCEparam2} by using the extended NCE.
	Assume that the distribution $p_*(x)=p(x \mid \xi^*)$ with $\xi^* = {\rm arg} \min_{\xi} d_{{\rm NCE}} (q,p_{\xi})$ has exactly $K$ mixture components: $\pi^*_1>0,\dots,\pi^*_K>0$ and $\theta^*_i \neq \theta^*_j \ (i \neq j)$. 
	In this case, the model is regular around $\xi^*$.
	Therefore, Lemma \ref{lem_bias} is valid and so ${\rm NCIC}_1$ in \eqref{NCIC0} is approximately unbiased.
	Also, by replacing $j_m(\xi^*)$ with $h = \sum_{l=m-K+1}^m j_l(\xi^*)$ in the proof, Lemma \ref{lem_bias2} for well-specified cases is valid as well, where the value of $m$ is changed from Section \ref{sec_NCIC} to $m={\rm dim}(\xi)=K ({\rm dim}(\theta_1)+1)$.
	Therefore, we propose
	\begin{align*}
		{\rm NCIC}_2 = N \hat{d}_{{\rm NCE}} (\hat{\xi}_{{\rm NCE}}) + K \left\{ {\rm dim}(\theta_1)+1 \right\} -\frac{1}{N+M} \left\{ \sum_{t=1}^N \hat{b}(x^{(t)}) + \sum_{t=1}^M \hat{b}(y^{(t)}) \right\}
	\end{align*}
	as an approximately unbiased estimator of $N {\rm E}_{x,y} \{ {d}_{{\rm NCE}} (q,\hat{p}) \}$, where $\hat{b}(z)$ is defined as \eqref{b_def}.
	Thus, we can select the number of components $K$ of non-normalized mixture models \eqref{NCEparam2} by minimizing NCIC.
	
	Figure~\ref{fig_gm} shows a result on the non-normalized version of the Gaussian mixture distribution:
	\begin{align}
		p(x \mid \theta,c) = \sum_{k=1}^K \exp (\theta_{k1} x^2 + \theta_{k2} x + c_k). \label{GMM}
	\end{align}
	Here, we generated $N=10^3$ samples from the two-component Gaussian mixture distribution $0.5 \cdot {\rm N} (0,1) + 0.5 \cdot {\rm N} (3,1)$ and applied the extended NCE to estimate \eqref{GMM}.
	The noise distribution was set to the Gaussian distribution with the same mean and variance as data and the noise sample size was set to $M=10^4$.
	For AIC, we computed the maximum likelihood estimator with the MATLAB function \textit{fitgmdist}.
	Both ${\rm NCIC}_2$ and AIC take minimum at the true value $K=2$.
	
	\begin{figure}
		\begin{minipage}{0.45\textwidth}
			\begin{center}
				\includegraphics[width=6cm]{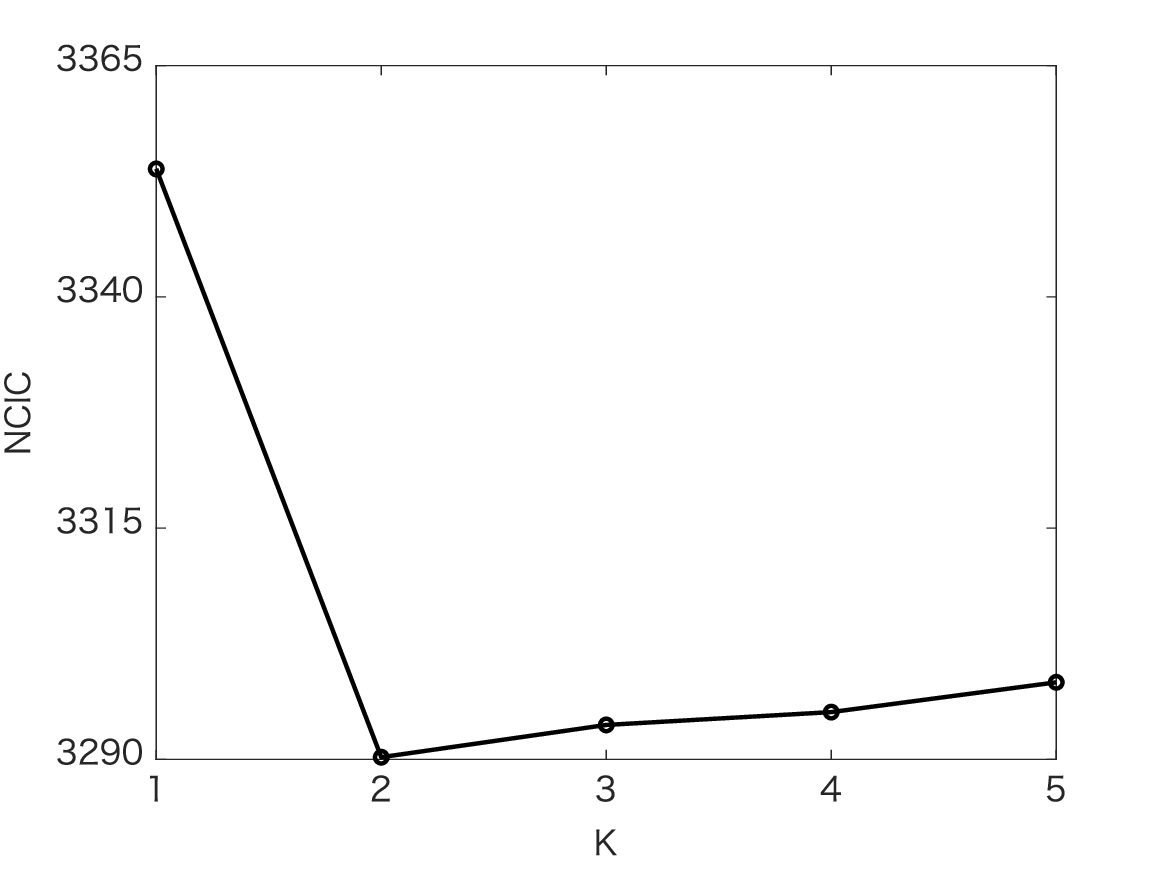}
			\end{center}
		\end{minipage}
		\begin{minipage}{0.45\textwidth}
			\begin{center}
				\includegraphics[width=6cm]{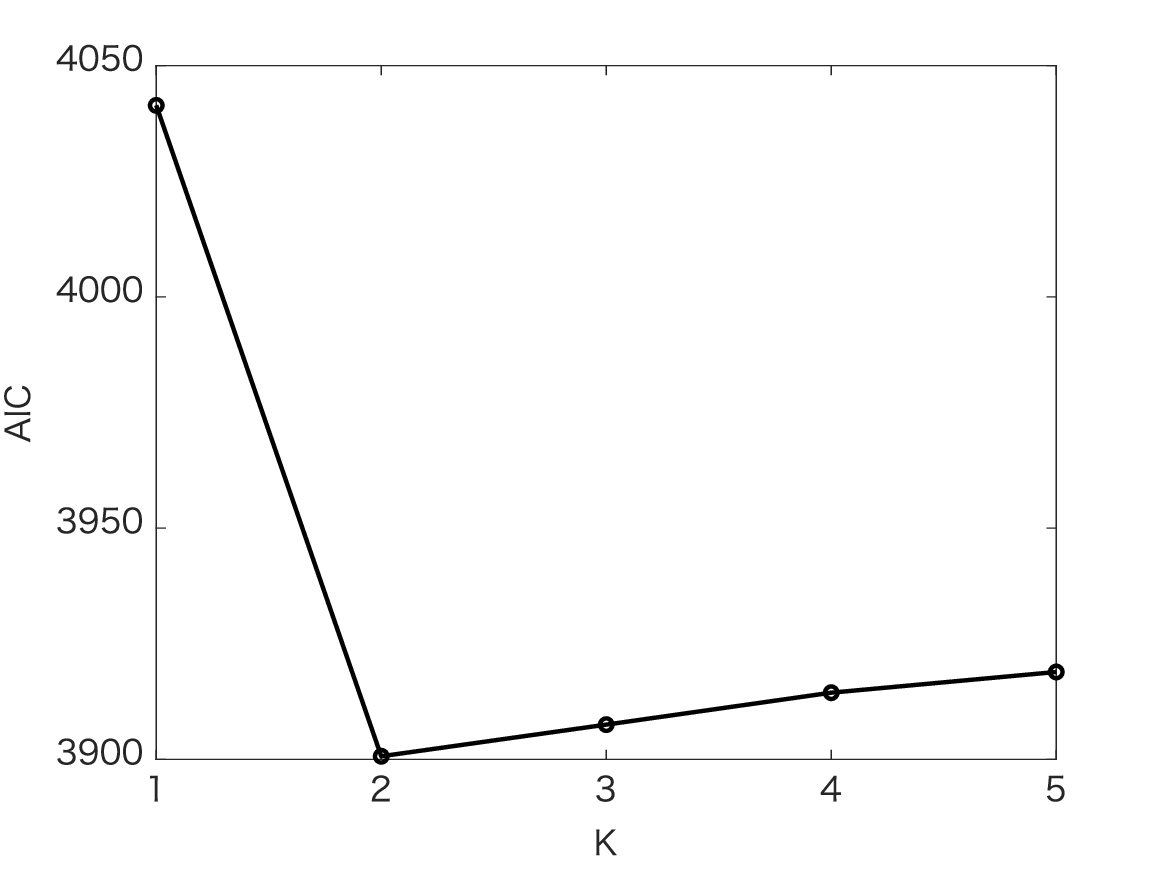}
			\end{center}
		\end{minipage}
		\caption{${\rm NCIC}_2$ (left) and ${\rm AIC}$ (right) for Gaussian mixture models. The true value is $K=2$.}
		\label{fig_gm}
	\end{figure}
	
	In the above, we assumed that the model is regular around $\xi^*$.
	It is not trivial to eliminate this condition due to the singularity in the parameter space of finite mixture models \citep{Mclachlan}.
	It is an interesting future work to develop a rigorous theory of model selection for non-normalized mixture models accounting for singularity \citep{Watanabe21}.

	\section{Conclusion}
	In this study, we developed information criteria for non-normalized models estimated by noise contrastive estimation (NCE) or score matching.
	The proposed criteria are approximately unbiased estimators of discrepancy measures for non-normalized models.
	They provide a principled method of model selection for general non-normalized models.
	We believe that this study increases the practicality of non-normalized models.
	
	Regarding future work, an interesting direction would be to apply NCIC to data-driven selection of neural network architectures \citep{Murata}, which is a constant problem in deep learning. 
	In a sense, the experiment of overcomplete independent component analysis on natural image data in Section~8.1  is viewed as selecting the number of units. 
	In a similar way, NCIC may be applicable to select the number of layers.
	Note that \cite{Gutmann12} applied NCE to train neural networks on natural image data.
	It would be also interesting if we can select from different architectures such as ResNet and CNN.
	Furthermore, since regularization is essential to avoid overfitting in training high-dimensional models including neural networks, it is an important problem to extend NCIC and SMIC to regularized cases such as LASSO \citep{Ninomiya}.
	On the other hand, since recent studies \citep{Fujikoshi,Yanagihara,Bai} have found that AIC attains consistency in high-dimensional settings (in contrary to low-dimensional settings), it is an interesting future work to investigate the consistency of NCIC and SMIC in high-dimensional settings.
	
	\section*{Acknowledgements}
We thank Yoshiyuki Ninomiya and Ricardo Monti for helpful comments. T.M. was supported by JSPS KAKENHI Grant Number 19K20220 and JST Moonshot Grant Number JPMJMS2024. A.H. was supported by a Fellowship from CIFAR.


\begin{thebibliography}{99}
		\bibitem[Akaike(1974)]{Akaike}
		\textsc{Akaike, H.} (1974).
		\newblock{A new look at the statistical model identification}.
		\newblock \textit{IEEE Transactions on Automatic Control}, \textbf{19}, 716--723.
		
		\bibitem[Bai et al.(2018)]{Bai}
		\textsc{Bai, Z.}, \textsc{Choi, K. P.} \& \textsc{Fujikoshi, Y.} (2018).
		\newblock{Consistency of AIC and BIC in estimating the number of significant components in high-dimensional principal component analysis}.
		\newblock \textit{The Annals of Statistics}, \textbf{46}, 1050--1076.
		
		\bibitem[Barp et al.(2019)]{Barp}
		\textsc{Barp, A.}, \textsc{Briol, F. X.}, \textsc{Duncan, A.}, \textsc{Girolami, M.} \& \textsc{Mackey, L.} (2019).
		\newblock{Minimum Stein discrepancy estimators}.
		\newblock In \textit{Advances in Neural Information Processing Systems 32 (NeurIPS 2019)}.
		
		\bibitem[Belloni and Chernozhukov(2013)]{Belloni}
		\textsc{Belloni, A.} \& \textsc{Chernozhukov, V.} (2013).
		\newblock{Least squares after model selection in high-dimensional sparse models}.
		\newblock \textit{Bernoulli}, \textbf{19}, 521--547.
		
		\bibitem[Besag(1974)]{Besag}
		\textsc{Besag, J.} (1974).
		\newblock{Spatial interaction and the statistical analysis of lattice systems}.
		\newblock \textit{Journal of the Royal Statistical Society B}, \textbf{36}, 192--236.
		
		\bibitem[{Burnham and Anderson(2002)}]{Burnham}
		\textsc{Burnham, K. P.} \& \textsc{Anderson, D. R.} (2002).
		\newblock \textit{Model Selection and Multimodel Inference}.
		\newblock New York: Springer.
		
		\bibitem[Caimo and Friel(2011)]{Caimo}
		\textsc{Caimo, A.} \& \textsc{Friel, N, J.} (2011).
		\newblock{Bayesian inference for exponential random graph models}.
		\newblock \textit{Social Networks}, \textbf{33}, 41--55.
		
		\bibitem[{Chikuse(2003)}]{Chikuse}
		\textsc{Chikuse, Y.} (2003).
		\newblock \textit{Statistics on Special Manifolds}.
		\newblock New York: Springer.
		
		\bibitem[{Claeskens and Hjort(2008)}]{Claeskens}
		\textsc{Claeskens, G.} \& \textsc{Hjort, N. L.} (2008).
		\newblock \textit{Model Selection and Model Averaging}.
		\newblock Cambridge: Cambridge University Press.
		
		\bibitem[{Dawid and Musio(2015)}]{Dawid}
		\textsc{Dawid, A. P.} \& \textsc{Musio, M.} (2015).
		\newblock{Bayesian model selection based on proper scoring rules}.
		\newblock \textit{Bayesian Analysis}, \textbf{10}, 479--499.
		
		\bibitem[Drton and Perlman(2004)]{Drton}
		\textsc{Drton, M.} \& \textsc{Perlman, M. D.} (2004).
		\newblock{Model selection for Gaussian concentration graphs}.
		\newblock \textit{Biometrika}, \textbf{91}, 591--602.
		
		\bibitem[Forbes and Lauritzen(2015)]{Forbes}
		\textsc{Forbes, P. G. M.} \& \textsc{Lauritzen, S.} (2015).
		\newblock{Linear estimating equations for exponential families with application to Gaussian linear concentration models}.
		\newblock \textit{Linear Algebra and its Applications}, \textbf{473}, 261--283.
		
		\bibitem[Fujikoshi et al.(2014)]{Fujikoshi}
		\textsc{Fujikoshi, Y.}, \textsc{Sakurai, T.} \& \textsc{Yanagihara, H.} (2014).
		\newblock{Consistency of high-dimensional AIC-type and Cp-type criteria in multivariate linear regression}.
		\newblock \textit{Journal of Multivariate Analysis}, \textbf{123}, 184--200.
		
		\bibitem[Gelman et al.(2014)]{Gelman}
		\textsc{Gelman, A.}, \textsc{Hwang, J.} \& \textsc{Vehtari, A.} (2014).
		\newblock{Understanding predictive information criteria for Bayesian models}.
		\newblock \textit{Statistics and Computing}, \textbf{24}, 997--1016.
		
		\bibitem[Geyer(1994)]{Geyer}
		\textsc{Geyer, C. J.} (1994).
		\newblock{On the convergence of Monte Carlo maximum likelihood calculations}.
		\newblock \textit{Journal of the Royal Statistical Society B}, \textbf{56}, 261--274.
		
		
		
		\bibitem[Grant and Boyd(2018)]{cvx}
		\textsc{Grant, M.} \& \textsc{Boys, S.} (2018).
		\newblock{CVX: Matlab software for disciplined convex programming}.
		\newblock version 2.1, December 2018.
		\url{http://cvxr.com/cvx}
		
		\bibitem[Gutmann and Hirayama(2011)]{Hirayama}
		\textsc{Gutmann, M. U.} \& \textsc{Hirayama, J.} (2011).
		\newblock{Bregman divergence as general framework to estimate unnormalized statistical models}.
		\newblock In \textit{Proceedings of the 27th Conference on Uncertainty in Artificial Intelligence (UAI 2011)}.
		
		\bibitem[{Gutmann and Hyv{\"a}rinen(2010)}]{Gutmann}
		\textsc{Gutmann, M. U.} \& \textsc{Hyv\"arinen, A.} (2010).
		\newblock{Noise-contrastive estimation: A new estimation principle for unnormalized statistical models}.
		\newblock In \textit{Proceedings of the 13th International Workshop on Artificial Intelligence and Statistics (AISTATS 2010)}.
		
		\bibitem[{Gutmann and Hyv{\"a}rinen(2012)}]{Gutmann12}
		\textsc{Gutmann, M. U.} \& \textsc{Hyv\"arinen, A.} (2012).
		\newblock{Noise-contrastive estimation of unnormalized statistical models, with applications to natural image statistics}.
		\newblock \textit{Journal of Machine Learning Research}, \textbf{13}, 307--361.
		
		\bibitem[Hinton(2002)]{Hinton}
		\textsc{Hinton, G. E.} (2002).
		\newblock{Training products of experts by minimizing contrastive divergence}.
		\newblock \textit{Neural Computation}, \textbf{14}, 1771--1800.
		
		\bibitem[{Hyv{\"a}rinen et al.(2001)}]{ICA_book}
		\textsc{Hyv\"arinen, A.}, \textsc{Karhunen, J.} \& \textsc{Oja, E.} (2001).
		\newblock \textit{Independent Component Analysis}.
		\newblock New York: Wiley.
		
		\bibitem[{Hyv{\"a}rinen(2005)}]{SM}
		\textsc{Hyv\"arinen, A.} (2005).
		\newblock{Estimation of non-normalized statistical models by score matching}.
		\newblock \textit{Journal of Machine Learning Research}, \textbf{6}, 695--709.
		
		\bibitem[{Hyv{\"a}rinen(2007)}]{SM2}
		\textsc{Hyv\"arinen, A.} (2007).
		\newblock{Some extensions of score matching}.
		\newblock \textit{Computational Statistics \& Data Analysis}, \textbf{51}, 2499--2512.
		
		\bibitem[{Hyv{\"a}rinen et al.(2009)}]{nis_book}
		\textsc{Hyv\"arinen, A.}, \textsc{Hurri, J.} \& \textsc{Hoyer, P. O.} (2009).
		\newblock \textit{Natural Image Statistics: A probabilistic approach to early computational vision}.
		\newblock New York: Springer.
		
		\bibitem[{Ji and Seymour(1996)}]{Ji}
		\textsc{Ji, C.} \& \textsc{Seymour, L.} (1996).
		\newblock{A consistent model selection procedure for Markov random fields based on penalized pseudolikelihood}.
		\newblock \textit{Annals of Applied Probability}, \textbf{6}, 423--443.
		
		\bibitem[{Kitagawa(1997)}]{Kitagawa97}
		\textsc{Kitagawa, G.} (1997).
		\newblock{Information criteria for the predictive evaluation of bayesian models}.
		\newblock \textit{Communications in Statistics - Theory and Methods}, \textbf{26}, 2223--2246.
		
		\bibitem[{Konishi and Kitagawa(1996)}]{Konishi96}
		\textsc{Konishi, S.} \& \textsc{Kitagawa, G.} (1996).
		\newblock{Generalised information criteria in model selection}.
		\newblock \textit{Biometrika}, \textbf{83}, 875--890.
		
		\bibitem[{Konishi and Kitagawa(2008)}]{Konishi08}
		\textsc{Konishi, S.} \& \textsc{Kitagawa, G.} (2008).
		\newblock \textit{Information Criteria and Statistical Modeling}.
		\newblock New York: Springer.
		
		\bibitem[{Lauritzen(1996)}]{Lauritzen}
		\textsc{Lauritzen, S. L.} (1996).
		\newblock \textit{Graphical Models}.
		\newblock Oxford: Oxford University Press.
		
		\bibitem[Li(2001)]{Li}
		\textsc{Li, S. Z.} (2001).
		\newblock \textit{Markov Random Field Modeling in Image Analysis}.
		\newblock New York: Springer.
		
		\bibitem[{Lin et al.(2016)}]{Lin}
		\textsc{Lin, L.}, \textsc{Drton, M.} \& \textsc{Shojaie, A.} (2016).
		\newblock{Estimation of high-dimensional graphical models using regularized score matching}.
		\newblock \textit{Electronic Journal of Statistics}, \textbf{10}, 806--854.
		
		\bibitem[Liu et al.(2019)]{Liu}
		\textsc{Liu, S.}, \textsc{Kanamori, T.}, \textsc{Jitkrittum, W.} \& \textsc{Chen, Y.} (2019).
		\newblock{Fisher efficient inference of intractable models}.
		\newblock In \textit{Advances in Neural Information Processing Systems 32 (NeurIPS 2019)}.
		
		
		\bibitem[Lyu(2009)]{Lyu}
		\textsc{Lyu, S.} (2009).
		\newblock{Interpretation and Generalization of Score Matching}.
		\newblock In \textit{Proceedings of the 25th International Conference on Uncertainty in Artificial Intelligence (UAI 2009)}.
		
		\bibitem[{Mardia and Jupp(2008)}]{Mardia}
		\textsc{Mardia, K. V.} \& \textsc{Jupp, P. E.} (2008).
		\newblock \textit{Directional Statistics}.
		\newblock New York: Wiley.
		
		
		\bibitem[Matsuda and Hyv{\"a}rinen(2019)]{Matsuda}
		\textsc{Matsuda, T.} \& \textsc{Hyv{\"a}rinen, A.} (2019).
		\newblock{Estimation of non-normalized mixture models}.
		\newblock In \textit{Proceedings of the 22nd International Conference on Artificial Intelligence and Statistics (AISTATS 2019)}.
		
		\bibitem[Mattheou et al.(2009)]{Mattheou}
		\textsc{Mattheou, K.}, \textsc{Leeb, S.} \& \textsc{Karagrigoriou, A.} (2001).
		\newblock{A model selection criterion based on the BHHJ measure of divergence}.
		\newblock \textit{Journal of Statistical Planning and Inference}, \textbf{139}, 228--235.
		
		\bibitem[{Mclachlan and Peel(2004)}]{Mclachlan}
		\textsc{Mclachlan, G.} \& \textsc{Peel, D.} (2000).
		\newblock \textit{Finite Mixture Models}.
		\newblock New York: Wiley.
		
		\bibitem[Murata et al.(1994)]{Murata}
		\textsc{Murata, N.}, \textsc{Yoshizawa, S.} \& \textsc{Amari, S.} (1994).
		\newblock{Network information criterion-determining the number of hidden units for an artificial neural network model}.
		\newblock \textit{IEEE Transactions on Neural Networks}, \textbf{5}, 865--872.
		
		
		\bibitem[{Ninomiya and Kawano(2016)}]{Ninomiya}
		\textsc{Ninomiya, Y.} \& \textsc{Kawano, S.} (2016).
		\newblock{AIC for the Lasso in generalized linear models}.
		\newblock \textit{Electronic Journal of Statistics}, \textbf{10}, 2537--2560.
		
		\bibitem[{Olshausen and Field(1997)}]{Olshausen}
		\textsc{Olshausen, B. A.} \& \textsc{Field, D. J.} (2001).
		\newblock{Sparse coding with an overcomplete basis set: A strategy employed by V1?}.
		\newblock \textit{Vision Research}, \textbf{37}, 3311--3325.
		
		\bibitem[{Pan(2001)}]{Pan}
		\textsc{Pan, W.} (2001).
		\newblock{Akaike's Information Criterion in Generalized Estimating Equations}.
		\newblock \textit{Biometrics}, \textbf{57}, 120--125.
		
		\bibitem[{Parry et al.(2012)}]{Parry}
		\textsc{Parry, M.}, \textsc{Dawid, A. P.} \& \textsc{Lauritzen, S.} (2012).
		\newblock{Proper local scoring rules}.
		\newblock \textit{The Annals of Statistics}, \textbf{40}, 561--592.
		
	\bibitem[{Qin(2001)}]{Qin}
\textsc{Qin, J.} (1998).
\newblock{Inferences for case-control and semiparametric two-sample density ratio models}.
\newblock \textit{Biometrika}, \textbf{85}, 619--630.

		\bibitem[Rasmussen(2006)]{Rasmussen}
		\textsc{Rasmussen, C. E.} (2006).
		\newblock{Conjugate gradient algorithm}.
		\newblock Matlab code version 2006-09-08. \url{http://learning.eng.cam.ac.uk/carl/code/minimize/minimize.m}
		
		\bibitem[{Ravikumar et al.(2010)}]{Ravikumar}
		\textsc{Ravikumar, P.}, \textsc{Wainwright, M. J.} \& \textsc{Lafferty, J. D.} (2010).
		\newblock{High-dimensional Ising model selection using $l_1$-regularized logistic regression}.
		\newblock \textit{The Annals of Statistics}, \textbf{38}, 1287--1319.
		
		\bibitem[Riou-Durand and Chopin(2018)]{Chopin}
		\textsc{Riou-Durand, L.} \& \textsc{Chopin, N.} (2018).
		\newblock{Noise contrastive estimation: Asymptotic properties, formal comparison with MC-MLE}.
		\newblock \textit{Electronic Journal of Statistics}, \textbf{12}, 3473--3518.
		
		\bibitem[Shao et al.(2019)]{Shao}
		\textsc{Shao, S.}, \textsc{Jacob, P. E.}, \textsc{Ding, J.} \& \textsc{Tarokh, V.} (2019).
		\newblock{Bayesian model comparison with the Hyvarinen score: computation and consistency}.
		\newblock \textit{Journal of the American Statistical Association}, \textbf{114}, 1826--1837.
		
		\bibitem[Singh et al.(2002)]{Singh}
		\textsc{Singh, H.}, \textsc{Hnizdo, V.} \& \textsc{Demchuk, E.} (2002).
		\newblock{Probabilistic model for two dependent circular variables}.
		\newblock \textit{Biometrika}, \textbf{89}, 719--723.
		
		\bibitem[Spiegelhalter et al.(2002)]{Spiegelhalter}
		\textsc{Spiegelhalter, D. J.}, \textsc{Best, N. G.}, \textsc{Carlin, B. P.} \& \textsc{van der Linde, A.} (2002).
		\newblock{Bayesian measures of model complexity and fit}.
		\newblock \textit{Journal of the Royal Statistical Society B}, \textbf{64}, 583--639.
		
		\bibitem[Stone(1977)]{Stone}
		\textsc{Stone, M.} (1977).
		\newblock{An asymptotic equivalence of choice of model by cross-validation and Akaike’s criterion}.
		\newblock \textit{Journal of the Royal Statistical Society B}, \textbf{39}, 44--47.
		
		
		\bibitem[Takeuchi(1976)]{Takeuchi}
		\textsc{Takeuchi, K.} (1976).
		\newblock{Distribution of information statistics and criteria for adequacy of models}.
		\newblock \textit{Mathematical Sciences}, \textbf{153}, 12--18 (in Japanese).
		
		\bibitem[Teh et al.(2004)]{Teh}
		\textsc{Teh, Y.}, \textsc{Welling, M.}, \textsc{Osindero, S.} \& \textsc{Hinton, G. E.} (2004).
		\newblock{Energy-based models for sparse overcomplete representations}.
		\newblock \textit{Journal of Machine Learning Research}, \textbf{4}, 1235--1260.
		
		\bibitem[Uehara et al.(2018)]{Uehara}
		\textsc{Uehara, M.}, \textsc{Matsuda, T.} \& \textsc{Komaki, F.} (2018).
		\newblock{Analysis of noise contrastive estimation from the perspective of asymptotic variance}.
		\newblock arXiv:1808.07983.
		
		\bibitem[Uehara et al.(2020a)]{Uehara20a}
		\textsc{Uehara, M.}, \textsc{Kanamori, T.}, \textsc{Takenouchi, T.} \& \textsc{Matsuda, T.}  (2020a).
		\newblock{A unified statistically efficient estimation framework for unnormalized models}.
		\newblock In \textit{Proceedings of the 23rd International Conference on Artificial Intelligence and Statistics (AISTATS 2020)}.
		
		\bibitem[Uehara et al.(2020b)]{Uehara20b}
		\textsc{Uehara, M.}, \textsc{Matsuda, T.}  \& \textsc{Kim, J. K.} (2020b).
		\newblock{Imputation estimators for unnormalized models with missing data}.
		\newblock In \textit{Proceedings of the 23rd International Conference on Artificial Intelligence and Statistics (AISTATS 2020)}.
		
		\bibitem[van der Vaart(1998)]{vV}
		\textsc{van der Vaart, A. W.} (1998).
		\newblock \textit{Asymptotic Statistics}.
		\newblock Cambridge University Press.
		
		\bibitem[{Varin and Vidoni(2005)}]{Varin}
		\textsc{Varin, C.} \& \textsc{Vidoni, P.} (2005).
		\newblock{A note on composite likelihood inference and model selection}.
		\newblock \textit{Biometrika}, \textbf{92}, 519--528.
		
		\bibitem[{Watanabe(2021)}]{Watanabe21}
		\textsc{Watanabe, S.} (2021).
		\newblock{WAIC and WBIC for mixture models}.
		\newblock \textit{Behaviormetrika}, \textbf{48}, 5--21.
		
		\bibitem[{Watanabe and Opper(2010)}]{Watanabe}
		\textsc{Watanabe, S.} \& \textsc{Opper, M.} (2010).
		\newblock{Asymptotic equivalence of Bayes cross validation and widely applicable information criterion in singular learning theory}.
		\newblock \textit{Journal of Machine Learning Research}, \textbf{11}, 3571--3594.
		
		\bibitem[Wooldridge(2001)]{Wooldridge}
		\textsc{Wooldridge, J. M.} (2001).
		\newblock{Asymptotic properties of weighted M-estimators for standard stratified samples}.
		\newblock \textit{Econometric Theory}, \textbf{17}, 451--470.
		
		\bibitem[Yanagihara et al.(2015)]{Yanagihara}
		\textsc{Yanagihara, H.}, \textsc{Wakaki, H.} \& \textsc{Fujikoshi, Y.} (2015).
		\newblock{A consistency property of the AIC for multivariate linear models when the dimension and the sample size are large}.
		\newblock \textit{Electronic Journal of Statistics}, \textbf{9}, 869--897.
		
		\bibitem[{Yu et al.(2018)}]{Yu18}
		\textsc{Yu, S.}, \textsc{Drton, M.} \& \textsc{Shojaie, A.} (2018).
		\newblock{Graphical models for non-negative data using generalized score matching}.
		\newblock In \textit{Proceedings of the 21st International Conference on Artificial Intelligence and Statistics (AISTATS 2018)}.
		
		\bibitem[{Yu et al.(2019)}]{Yu}
		\textsc{Yu, S.}, \textsc{Drton, M.} \& \textsc{Shojaie, A.} (2019).
		\newblock{Generalized score matching for non-negative data}.
		\newblock \textit{Journal of Machine Learning Research}, \textbf{20}, 1--70.
	\end{thebibliography}
\end{document}